\documentclass[11pt,oneside,english]{amsart}
\usepackage[T1]{fontenc}
\usepackage[utf8]{inputenc}
\usepackage{xcolor}
\usepackage{babel}
\usepackage{amstext}
\usepackage{amsthm}
\usepackage{amssymb}
\usepackage{graphicx}
\usepackage[unicode=true,pdfusetitle,
 bookmarks=true,bookmarksnumbered=false,bookmarksopen=false,
 breaklinks=false,pdfborder={0 0 1},backref=false,colorlinks=true]
 {hyperref}
\hypersetup{
 citecolor=blue}

\makeatletter
\numberwithin{equation}{section}
\numberwithin{figure}{section}
\theoremstyle{remark}
\newtheorem*{rem*}{\protect\remarkname}
\theoremstyle{plain}
\newtheorem*{cor*}{\protect\corollaryname}
\theoremstyle{plain}
\newtheorem*{question*}{\protect\questionname}
\theoremstyle{plain}
\newtheorem{thm}{\protect\theoremname}[section]
\theoremstyle{remark}
\newtheorem{rem}[thm]{\protect\remarkname}
\theoremstyle{plain}
\newtheorem{lem}[thm]{\protect\lemmaname}
\theoremstyle{plain}
\newtheorem{prop}[thm]{\protect\propositionname}
\theoremstyle{plain}
\newtheorem{cor}[thm]{\protect\corollaryname}
\theoremstyle{definition}
\newtheorem{defn}[thm]{\protect\definitionname}

\usepackage{babel}

\usepackage{geometry}
\usepackage{amsmath}

\numberwithin{equation}{section}
\numberwithin{figure}{section}
\usepackage{enumitem}		
 \let\footnote=\endnote
\@ifundefined{lettrine}{\usepackage{lettrine}}{}

\theoremstyle{theorem}
\newtheorem{thmx}{Theorem}

\def\d{\delta}

\def\R{\mathbb{R}}

\def\N{\mathbb{N}}

\def\ep{\varepsilon}
\def\vphi{\varphi}

\def\cal{\mathcal}

\def\T{\mathcal{T}}
\def\F{\mathcal{F}}

\usepackage{float}

\subjclass[2000]{}
\thanks{
Chen was partially supported by startup funding from IU Indianapolis (account No. 23-937-10), and Kao was partially supported by Simons Foundation grant No. 956047. A portion of this work was carried out during Kao's visit to the National Center for Theoretical Sciences (NCTS) in Taiwan, and he gratefully acknowledges NCTS for their support and hospitality. }

\def\J{\mathcal{J}}

\def\v{\mathsf{v}}

\def\R{\mathbb{R}}

\def\T{\mathcal{T}}
\def\CC{\mathcal{C}}

\def\ep{\varepsilon}

\def\T{\mathcal{T}}
\def\F{\mathcal{F}}

\def\Sing{\mathrm{Sing}}
\def\Reg{\mathrm{Reg}}
\def\vp{\varphi}

\def\wt{\tilde}
\newcommand{\II}{\text{II}}

  \providecommand{\corollaryname}{Corollary}
  \providecommand{\definitionname}{Definition}
  \providecommand{\lemmaname}{Lemma}
  \providecommand{\propositionname}{Proposition}
  \providecommand{\remarkname}{Remark}
  \providecommand{\theoremname}{Theorem}

\providecommand{\questionname}{Question}

\author[Chen]{Dong Chen}
\address{Indiana University Indianapolis, Indianapolis, IN 46202, USA}
\email{dc45@iu.edu}
\author[Kao]{Lien-Yung Kao}
\address{George Washington University, Washington, D.C. 20052, USA}
\email{lkao@gwu.edu}

\author[Park]{Kiho Park}
\email{kiho.park12@gmail.com}

\dedicatory{Dedicated to the memory of Todd Fisher}

\@ifundefined{showcaptionsetup}{}{%
 \PassOptionsToPackage{caption=false}{subfig}}
\usepackage{subfig}
\makeatother

\providecommand{\corollaryname}{Corollary}
\providecommand{\definitionname}{Definition}
\providecommand{\lemmaname}{Lemma}
\providecommand{\propositionname}{Proposition}
\providecommand{\questionname}{Question}
\providecommand{\remarkname}{Remark}
\providecommand{\theoremname}{Theorem}

\begin{document}
\title[Pressure gaps and geometric potentials]{ Pressure gaps, Geometric potentials, and nonpositively curved manifolds }
\begin{abstract}
In this paper, we derive a general pressure gap criterion for closed rank 1 manifolds with singular sets characterized by codimension 1 totally geodesic flat subtori. As an application, we demonstrate that under specific curvature constraints, potentials that decay faster than geometric potentials (towards the singular set) exhibit pressure gaps and lack phase transitions. Additionally, we prove that geometric potentials are Hölder continuous near singular sets.
\end{abstract}

\maketitle

\section{Introduction. }

This paper is dedicated to characterizing pressure gaps in nonuniformly hyperbolic dynamical systems originating from geometric contexts. Broadly speaking, pressure gaps can be interpreted as indicating that the ``magnitude'' of ``nonuniform hyperbolicity'', as measured by the topological pressure from the perspective of a potential, is small. It has been demonstrated that the presence of a pressure gap is crucial for a potential and its equilibrium states to exhibit ergodic properties similar to those in uniformly hyperbolic systems, including uniqueness, equidistribution, and the Bernoulli property.

The work of Burns, Climenhaga, Fisher, and Thompson \cite{burns2018unique} is one of the pioneering studies on pressure gaps in a geometric setting (see the historical remarks in this section). However, the characterization of pressure gaps remains incomplete. This paper aims to further investigate pressure gaps within natural and concrete geometric contexts.

Let $M$ be a closed, connected, smooth $n$-dimensional manifold, and $g$ be a $C^{m+2}$ ($m>0$) rank 1 Riemannian metric on $M$. Let $\mathcal{F}=(f_{t})_{t\in\R}$ be the geodesic flow on the unit tangent bundle of the Riemannian manifold $(M,g)$. The \textit{topological pressure} $P(\varphi)$ of a potential $\varphi$ is the supremum of the free energy $h_{\mu}({\cal F})+\int\varphi\, d\mu$ over ${\cal F}$-invariant Borel probability measures, where $h_{\mu}({\cal F})$ is the measure-theoretic entropy. A measure that achieves this supremum is called an \textit{equilibrium state} of $\varphi$.

For the geodesic flow ${\cal F}$, ``nonuniform hyperbolicity'' arises from a geometric object on $T^{1}M$, namely, the \textit{singular set} $\mathrm{Sing}$ (see Section \ref{sec:Preliminary} for the precise definition). A potential $\varphi$ is said to have a \textit{pressure gap} if $P(\mathrm{Sing},\varphi)<P(\varphi)$, where $P(\mathrm{Sing},\varphi)$ is the pressure restricted to the ${\cal F}$-invariant set $\mathrm{Sing}$. In this paper, we prove that if $\varphi$ decays rapidly enough near the singular set, then $\varphi$ has a pressure gap.

\subsection*{Setting}

We shall introduce notations and the setup to contextualize our results. Throughout the paper, we assume the following conditions on the Riemannian manifold $M$:

\begin{enumerate}
\item[(C1)] $M$ is a closed rank 1 manifold with nonpositive sectional curvature.
\item[(C2)] $\mathrm{Sing}$ is either $T^{1}T_{0}$ or the \textit{flat strip} case $T^{1}T_{0}\times[-1,1]$, where $T_{0}$ is a codimension 1 totally geodesic flat torus.
\item[(C3)] $M$ is negatively curved outside $T_{0}$ or the flat strip $T_{0}\times[-1,1]$.
\end{enumerate}

For these types of manifolds, it is convenient to study the behavior of geodesics near $T_{0}$ or $T_{0}\times[-1,1]$ using \textit{Fermi coordinates} (see Section \ref{sec:codim1} for more details). For $p\in M$, let $x(p)$ be the signed distance from $p$ to $T_{0}$ or $T_{0}\times[-1,1]$, and $s(p)$ the closest point on $T_{0}$ or $T_{0}\times[-1,1]$ from $p$. The map $s:M\to T_{0}$ is a projection onto $T_{0}$, which induces a projection \textcolor{black}{$ds:TM\to TT_{0}$.} 

For $v\in T_{p}M$, we consider two components from its Fermi coordinates: $x_{v}=x(p)$, and the signed angle $\phi_{v}$ between $v$ and the hypersurface $x=x(p)$. We define the \textit{radial curvature} $K_{\perp}(v)$ at $v$ as the sectional curvature of the tangent plane spanned by $\{v,  X\}$ where $X=\partial/\partial x$.

We are particularly interested in the following two types of manifolds:
\begin{itemize}
\item $M$ is \textit{Type 1} if $M$ has an order $m$ \textbf{uniform} curvature bound, i.e., $K_{\perp}$ vanishes uniformly to order $m-1$ over $T_{0}$ or the flat strip (see (\ref{eq: uniform curvature control}) for more details).
\item $M$ is \textit{Type 2} if $\dim M=2$ and $M$ has an order $m$ \textbf{nonuniform} curvature bound, i.e., the Gaussian curvature vanishes up to order $m-1$ over $T_{0}$ or the flat strip (see (\ref{eq: nonuniform-1}) and (\ref{eq: nonuniform-2})).
\end{itemize}
In Examples subsection \ref{examples}, we discuss manifolds satisfying the above hypotheses in more detail. For example, surfaces of genus greater than one with analytic Riemannian metrics are Type 2 manifolds. See Figure \ref{fig:type1} and Figure \ref{fig:type2} for some illustrations.

\begin{figure}[h]
\subfloat[No flat strips]{\includegraphics[scale=0.38]{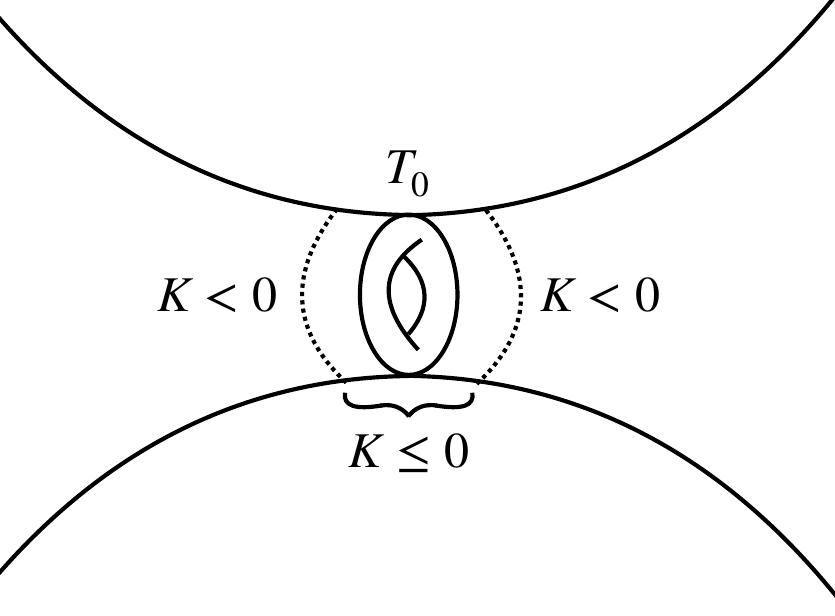}\label{fig:type1-wo-fs}
}\qquad{}\subfloat[With a flat strip]{\includegraphics[scale=0.38]{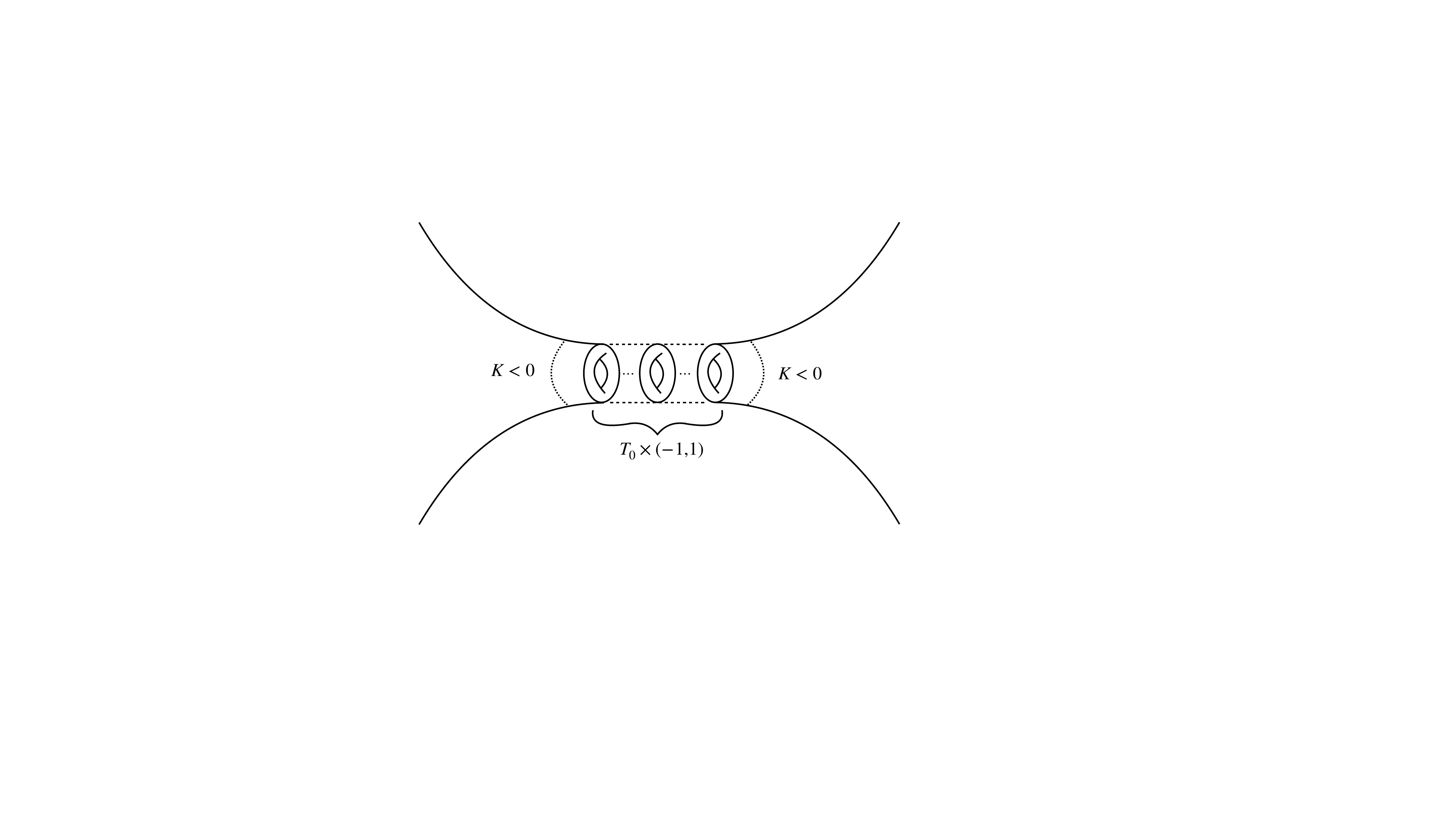}\label{fig:type1-w-fs}
}
\caption{Type 1 manifolds}
\label{fig:type1} 
\end{figure}

\begin{figure}[h]
\subfloat[No flat strips]{\includegraphics[scale=0.38]{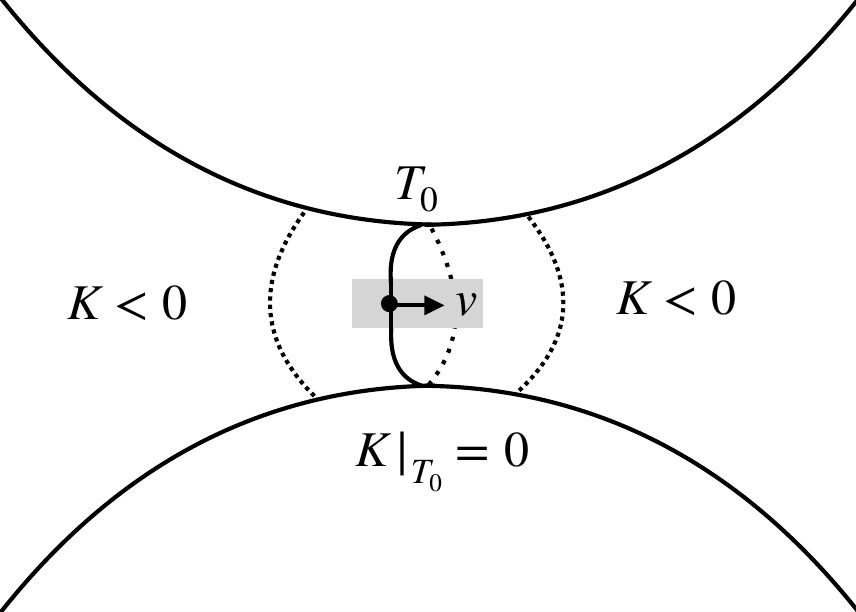}\label{fig:type2-wo-fs}
}\qquad{}\subfloat[With a flat strip]{\includegraphics[scale=0.38]{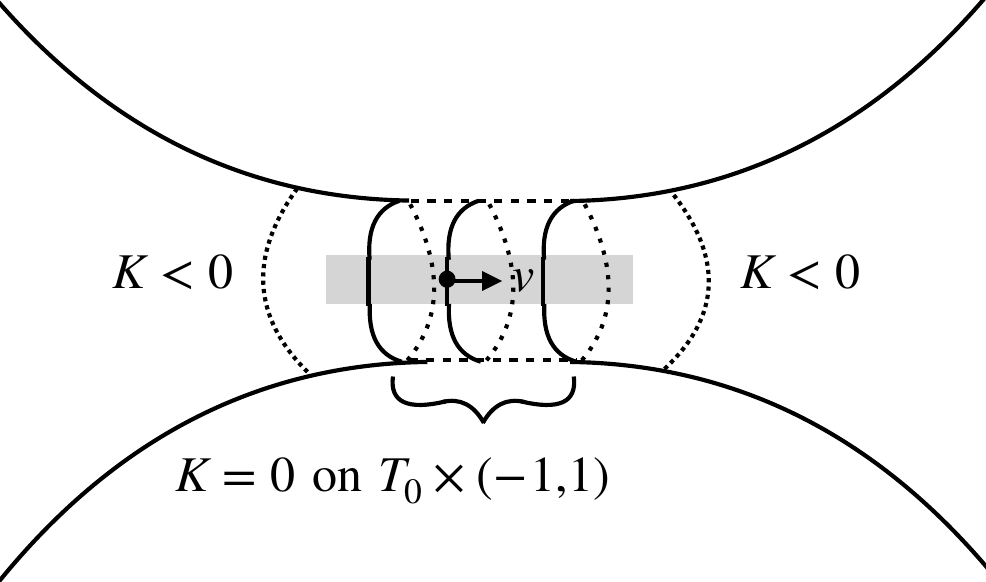}\label{fig:type2-w-fs}
}
\caption{Type 2 surfaces}
\label{fig:type2} 
\end{figure}

In this paper, in addition to conditions (C1)-(C3), we assume the following condition on the continuous potential $\varphi:T^{1}M\to\R$:
\begin{enumerate}
\item[(C4)] If $M$ is Type 1, then $\varphi$ is constant on $\mathrm{Sing}$. If $M$ is Type 2, then $\varphi$ is transversally constant on $\mathrm{Sing}$, meaning $\varphi$ depends only on the image of the projection $\mathrm{Sing} \to T^{1}T_{0}$.
\end{enumerate}
Note that when $M$ is Type 2, we do not assume that $\varphi$ is constant on $\mathrm{Sing}$.

Any potential function $\varphi:T^{1}M \to \mathbb{R}$ can be extended to $\varphi:TM \to \mathbb{R}$ via
\begin{equation}\label{eqn: rescaling_varphi}
\varphi(v) = |v| \varphi\left(\frac{v}{|v|}\right).
\end{equation}
With this extension, the integral of $\varphi$ along any curve is independent of parametrization.

\subsection*{General results for pressure gaps:}
We are now ready to state our first result on obtaining the pressure gap. We denote by $N_{R}(\mathrm{Sing})$ the radius $R$ neighborhood of $\mathrm{Sing}$ in $T^{1}M$ with respect to the Sasaki metric (see more details below).

\begin{thmx}[Pressure gap criterion]\label{Main-result}
Suppose $M$ is a manifold of Type 1 or 2 and $\varphi$ satisfy conditions \textup{(C1)-(C4)}, and there exist $R, C, \varepsilon_{1}, \varepsilon_{2} > 0$ such that
\begin{equation}
\varphi(v) - \varphi(ds(v))\geq -C\left(|x_{v}|^{\frac{m}{2} + \varepsilon_{1}} + |\phi_{v}|^{\frac{m}{m+2} + \varepsilon_{2}}\right)\label{eq:main-ineq}
\end{equation}
for any $v \in N_{R}(\mathrm{Sing})$. Then $P(\mathrm{Sing}, \varphi) < P(\varphi)$.
\end{thmx}

\textcolor{black}{In particular, any potential that is Hölder continuous with sufficiently large exponents at $\mathrm{Sing}$ satisfies \eqref{eq:main-ineq}. We also note that it is sufficient to have a lower bound on $\varphi(v) - \varphi(ds(v))$ near $\mathrm{Sing}$, since when $\varphi(v)$ is larger than $\varphi(ds(v))$, $\varphi$ accumulates more pressure outside the singular set. When $\varphi$ is locally constant, we achieve \cite[Theorem B]{burns2018unique} in our setup: }

\begin{cor*}[Locally constant potentials]
Suppose $M$ is Type 1 or Type 2, and $\varphi$ is locally constant in a neighborhood of $\mathrm{Sing}$, then $P(\varphi) > P(\mathrm{Sing}, \varphi)$. 
\end{cor*}

\begin{rem}
We briefly discuss the improvements and differences between Theorem \ref{Main-result} and \cite[Theorem B]{burns2018unique}:
\begin{enumerate}
\item We do not assume that $\varphi$ is constant on $\mathrm{Sing}$ when $M$ is Type 2.
\item One key difference is the construction of shadowing orbits. \cite{burns2018unique} uses stable and unstable manifolds to construct orbit segments that shadow those contained in the singular set. This method applies to all nonpositively curved manifolds but provides less control. Our approach requires more precise estimates of the shadowing orbits. To achieve this, we use bouncing orbits (see Figure \ref{fig:vec}) to shadow singular orbits, with curvature bounds providing additional control. See Section \ref{subsec:Shadowing-map} for more details.
\item Another difference is that the topological entropy of our singular set is zero, allowing us to better estimate $P(\mathrm{Sing}, \varphi)$. This is the primary reason we can dispense with the locally constant assumption in \cite[Theorem B]{burns2018unique}. However, when $M$ has a flat strip, the transversally constant condition in Theorem \ref{Main-result} is necessary because increasing $\varphi$ in the middle of the strip could eliminate the pressure gap. See \cite[Section 10.1]{burns2018unique} for more details.
\end{enumerate}
\end{rem}

\begin{rem}
For brevity and readability, we assume $M$ contains only one flat torus $T_{0}$ or one flat strip. However, Theorem \ref{Main-result} holds when $\mathrm{Sing}$ is induced by finitely many codimension 1 totally geodesic flat tori or flat strips. Specifically, Theorem \ref{Main-result} is valid under the following assumptions:
\begin{enumerate}
\item $\mathrm{Sing} = \coprod_{i=0}^{k} \mathsf{S}_{i}$, where $\mathsf{S}_{i} = T^{1}T_{i}$ or $T^{1}T_{i} \times [-1,1]$ for a codimension 1 totally geodesic flat subtorus $T_{i}$.
\item Without loss of generality, we assume $P(\varphi|_{\mathsf{S}_{0}}) = \max\{P(\varphi|_{\mathsf{S}_{i}}) : i = 0, \ldots, k\} = P(\mathrm{Sing}, \varphi)$. We only need $\varphi$ to satisfy the assumptions of Theorem \ref{Main-result} near $\mathsf{S}_{0}$.
\item $M$ satisfies the curvature bounds near $T_{0}$ or $T_{0} \times [-1,1]$ as a Type 1 or Type 2 manifold.

we do not know which one
\end{enumerate}
\end{rem}

We say that a potential $\varphi$ has a \textit{phase transition} at $q_{0}$ if the pressure map $q \mapsto P(q\varphi)$ fails to be differentiable at $q_{0}$. It is well known that the uniqueness of equilibrium states implies the differentiability of the pressure map; see \cite{Ruelle1978Diff}.

Our second main result is that no phase transition appears if $\varphi$ decays rapidly near $\mathrm{Sing}$:

\begin{thmx}[No phase transition]\label{thm:B}
Let $M$ and $\varphi$ satisfy conditions \textup{(C1)-(C4)}, and let $\varphi$ be a Hölder continuous potential such that 
\begin{equation}
\left|\varphi(v) - \varphi(ds(v))\right| \leq C \left(|x_{v}|^{\frac{m}{2} + \varepsilon_{1}} + |\phi_{v}|^{\frac{m}{m+2} + \varepsilon_{2}}\right), \quad \forall\,v \in N_{R}(\mathrm{Sing}).\label{eq:main-eq-2}
\end{equation}
Then $q\varphi$ has a unique equilibrium state for each $q \in \mathbb{R}$; thus, $\varphi$ does not have phase transitions.
\end{thmx}

\begin{rem}
\begin{enumerate}
\item The argument for Theorem \ref{thm:B} remains valid even if $\varphi$ is only controlled from below, i.e., satisfying (\ref{eq:main-ineq}). In this case, the conclusion is that $q\varphi$ has a unique equilibrium state for $q > 0$.
\item Theorem \ref{thm:B} follows immediately from Theorem \ref{Main-result}.
\end{enumerate}
\end{rem}





The proof of Theorem \ref{Main-result} relies on an abstract pressure criterion (Proposition \ref{prop: abstract pressure gap}). For readability, we defer the detailed exposition of this abstract result to Section \ref{sec: pressure gap}. In essence, Proposition \ref{prop: abstract pressure gap} demonstrates that if the geodesic flow $\mathcal{F}$ satisfies the following conditions: (1) a ``strong'' specification property, (2) the singular set can be shadowed by nearby vectors, and (3) the potential $\varphi$ decays rapidly enough, then $\varphi$ exhibits a pressure gap.

\subsection*{Results for Geometric Potentials}

The second theme of this paper is dedicated to studying the behavior of the geometric potential $\varphi^{u}: T^{1}M \to \mathbb{R}$ near $T_{0}$. Recall that the geometric potential $\varphi^{u}$ is defined as
\[
\varphi^{u}(v) := -\lim_{t \to 0} \frac{1}{t} \log \det(df_{t} \mid_{E^{u}(v)})
\]
where $E^{u}(v)$ is the unstable subspace (see Section \ref{sec:Preliminary} for details).

For Type 2 surfaces without flat strips, Gerber and Niţică \cite[Theorem 3.1]{gerber1999ETDS} and Gerber and Wilkinson \cite[Lemma 3.3]{gerber1999holder} provided Hölder continuity estimates for the geometric potential $\varphi^{u}$ at $\mathrm{Sing}$. The following result shows that, under a natural Ricci curvature constraint, similar Hölder continuity estimates can be extended to higher-dimensional cases.

In what follows, we denote by $a(v) \approx b(v)$ near $S$, if there exists a neighborhood $N$ of $S$ and $C > 1$ such that $C^{-1}b(v) \leq a(v) \leq Cb(v)$ for $v \in N$. We say $M$ has order $m$ Ricci curvature bounds if the Ricci curvature $\mathrm{Ric}(v)$ vanishes uniformly to order $m-1$ over $T_{0}$ (see \eqref{eqn: ric_curv_bdd}).

\begin{thmx}[Geometric Potentials] \label{Main-result-2}
Let $M$ be a Type 1 manifold without flat strips. Suppose $M$ has an order $m$ Ricci curvature bound, then near $\mathrm{Sing}$ we have
\[
-\varphi^{u}(v) \approx |x_{v}|^{\frac{m}{2}} + |\phi_{v}|^{\frac{m}{m+2}}.
\]
\end{thmx}

The no-flat-strip condition is necessary for the Hölder continuity. See Remark \ref{rem: no_flat_strip_cond} for more details.

\begin{rem}
In general, radial curvature and Ricci curvature have no strong relationship. Only in the surface case are these two curvatures the same. Nevertheless, in the appendix, we show that if the Riemannian metric is a warped product, then the Ricci curvature bound and the radial curvature bound hypotheses are equivalent.
\end{rem}

In most cases, the Hölder continuity of geometric potentials is unknown for nonuniformly hyperbolic systems, especially in higher dimensions. As an immediate consequence of Theorem \ref{Main-result-2}, we have the following partial result for higher-dimensional manifolds:

\begin{thmx}[Local Hölder Continuity] \label{Main-result-3}
Under the same assumptions as Theorem \ref{Main-result-2}, the geometric potential $\varphi^{u}$ is Hölder continuous at $\mathrm{Sing}$.
\end{thmx}

We note that our method, inspired by \cite{gerber1999holder}, currently only establishes the Hölder continuity of the geometric potential $\varphi^{u}$ at $\mathrm{Sing}$. Achieving global Hölder continuity of $\varphi^{u}$ may require the Hölder continuity of the unstable Jacobian tensor $U^{u}$, which is still unclear in higher-dimensional cases.

On the other hand, as a consequence of Theorem \ref{Main-result-2}, we know that $\varphi^{u}$ is a borderline case of the pressure gap criterion given in Theorem \ref{Main-result} (i.e., $\varepsilon_{1} = \varepsilon_{2} = 0$). Moreover, it is known that for manifolds (including surfaces) whose singular sets are unit tangent bundles of flat, totally geodesic codimension 1 tori, the geometric potential $\varphi^{u}$ exhibits a phase transition at $q = 1$ (see \cite[p. 530]{burns2021phase} and \cite[Theorem C]{burns2018unique}).

In other words, Theorem \ref{Main-result-2} shows that the pressure gap criterion given in Theorem \ref{Main-result} is optimal in the sense that there are examples at the boundary of our criterion that do not have pressure gaps (see Figure \ref{fig:boundary}).

\begin{figure}[h]
\includegraphics[scale=0.8]{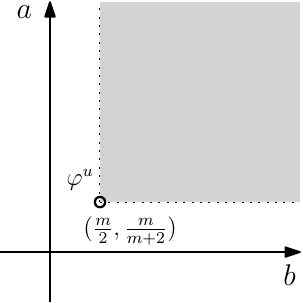}
\caption{Region where potentials $|\varphi(v)| \approx |x_{v}|^{a} + |\phi_{v}|^{b}$ exhibit a pressure gap and no phase transitions.}
\label{fig:boundary}
\end{figure}

In Figure \ref{fig:boundary}, each point corresponds to potentials $\varphi$ satisfying $-\varphi(v) \approx |x_{v}|^{a} + |\phi_{v}|^{b}$ near $\mathrm{Sing}$. The shaded region represents potentials that have a pressure gap and no phase transitions by Theorem \ref{Main-result}. The geometric potential $\varphi^{u}$ lies at the vertex $(\frac{m}{2}, \frac{m}{m+2})$ of the shaded region.

We conclude this subsection by posing an open question:

\begin{question*}
Suppose $(a, b)$ lies on the boundary of the shaded region in Figure \ref{fig:boundary}. Is there a potential $\varphi$ satisfying $|\varphi(v)| \approx |x_{v}|^{a} + |\phi_{v}|^{b}$ near $\mathrm{Sing}$ such that $\varphi$ has a phase transition?
\end{question*}

\subsection*{Examples: \label{examples}}

The prototype of Type 1 manifolds is the surface of revolution with profile \( f(x) = 1 + |x|^{r} \), which is the main example discussed in Lima, Matheus, and Melbourne \cite{lima2021polydecay}. For higher dimensions, an important example is the Heintze example (see Ballman, Brin, and Eberlein \cite{ballman1985nonpositive} or \cite[Section 10.2]{burns2018unique}). The simplest version of the Heintze example starts with a finite volume hyperbolic 3-manifold with one cusp, then removes the cusp and flattens the region near the cross-section. Recall that the cross-section of the cusp is a codimension 1 totally geodesic flat torus. The Heintze example is obtained by gluing two identical copies of the above 3-manifold along the cross-section (see Figure \ref{fig:type1}).

Type 2 surfaces (see Figure \ref{fig:type2}) were introduced in Gerber and Niţică \cite{gerber1999ETDS} and Gerber and Wilkinson \cite{gerber1999holder}. The archetype is a rank 1 nonpositively curved surface with an analytic metric. In such cases, it is well-known that $\mathrm{Sing}$ consists of unit tangent bundles of finitely many closed geodesics (see \cite[Section 10.1]{burns2018unique} for a sketched proof).

We remark that for nonpositively curved surfaces, Coudène and Schapira \cite[Theorem 3.2]{Coudene2014} (inspired by the unpublished work of Cao and Xavier \cite{cao2008flatstrip}) showed that flat strips for nonpositively curved surfaces close up. Moreover, Constantine, Lafont, McReynolds, and Thompson \cite[Theorem B]{MR3961216} establish a similar result in higher dimensions. However, in general, it is unknown if the singular geodesics or the (higher dimensional) zero curvature strips close up. Nevertheless, in all known examples, to the best of the authors' knowledge, the singular sets do close up, leading to our hypothesis on the existence of $T_{0}$.

\subsection*{Historical Remarks}

There is no singular set when the dynamical system is uniformly hyperbolic. Hence, the pressure gap persists for a broad class of potentials, allowing one to derive ergodic properties for associated equilibrium states. The origin of this fact traces back to the work of Bowen \cite{bowen1974some} for maps and Franco \cite{franco1977flows} for flows. For nonuniformly hyperbolic systems, Climenhaga and Thompson \cite{Climenhaga2016Uniqueness} proposed using the pressure gap as a condition to obtain ergodic properties of equilibrium states, particularly uniqueness, similar to uniform hyperbolic cases.

\textcolor{black}{Burns et al. \cite{burns2018unique} applied the argument from \cite{Climenhaga2016Uniqueness} and derived a necessary condition for the pressure gap.} Specifically, they showed that for closed rank 1 nonpositively curved manifolds, if $\varphi$ is locally constant near $\mathrm{Sing}$, then $\varphi$ has a pressure gap. This work was inspired by Knieper \cite{knieper1998uniqueness}, where the entropy gap $h_{\mathrm{top}}(\mathrm{Sing}) < h_{\mathrm{top}}(\mathcal{F})$ was established as a consequence of the uniqueness of the measure of maximal entropy. Recall that the topological entropy $h_{\mathrm{top}}(\mathcal{F})$ is the pressure of the zero potential.

\textcolor{black}{Gelfert and Schapira \cite{Gelfert:2014hn} compared different notions of pressure for closed rank 1 nonpositively curved manifolds, such as topological pressure, Gurevich pressure (or periodic orbit pressure), and their restrictions on singular and regular sets. They pointed out that under certain conditions, these different notions of pressure are identical. Similar discussions can be found in \cite[Propositions 2.8 and 6.4]{burns2018unique}.}

In geometry, the Liouville measure is an equilibrium state for the geometric potential. Ergodic properties of equilibrium states have been extensively studied. Several recent contributions have employed the Climenhaga-Thompson strategy (see \cite{Climengaha2021survey} for a survey on the strategy). For example, Chen, Kao, and Park \cite{chen2020unique,chen2021properties} worked on no focal points settings; Climenhaga, Knieper, and War focused on no conjugate points manifolds \cite{Climenhaga2021noconjugate}; the work of Call, Constantine, Erchenko,  Sawyer, and Work \cite{Call2021flatsurfaces} discussed flat surfaces with singularities.

There are many other relevant discussions on the ergodic properties of equilibrium states. \textcolor{black}{For example, uniqueness is addressed in \cite{Gelfert2019mme}; the Kolmogorov property in \cite{call2022K}; the Bernoulli property in \cite{Pesin1977Bernoulli,Burns1989Bernoulli,Ornstein1998Bernoulli,ledrappier2016Bernulli,call2022K,Lima2020Bernulli}; the central limit theorem in \cite{Thompson2021CLT,lima2021polydecay}; and the fact that there are at most countably many measures of maximal entropy for certain 3-dimensional flows in \cite{Lima-Sarig19-3d-symb}.}

For geometric potentials of rank 1 surfaces, Burns and Gelfert \cite{burns2014lyapunov} pointed out the existence of a phase transition. This was also confirmed in \cite{burns2018unique} using a different approach. A recent work by Burns, Buzzi, Fisher, and Sawyer \cite{burns2021phase} further investigated the edge case of $\varphi = q\varphi^{u}$ for $q = 1$. They showed that the Liouville measure is the only equilibrium state not supported on $\mathrm{Sing}$. However, the Hölder continuity of $\varphi^{u}$ is even less known. It is only assured by Gerber and Wilkinson \cite[Theorem I]{gerber1999holder} for Type 2 surfaces. Much less is known about $\varphi^{u}$ in higher dimensional cases.

\subsection*{Outline of the Paper}

In Section \ref{sec:Preliminary}, we recall some relevant background material from geometry and dynamics. In Section \ref{sec:Preparation}, we introduce the main shadowing technique and a key inequality to prove Theorem \ref{Main-result}. Sections \ref{sec: higher dim} and \ref{sec: GW surface} are devoted to technical estimates in Type 1 and Type 2 settings by analyzing the relevant Riccati equations. Section \ref{sec:geo_potentials} focuses on the geometric potential and the proof of Theorem \ref{Main-result-2}. The proof of Theorem \ref{Main-result} is presented in Section \ref{sec: pressure gap} as a consequence of a more general pressure gap criterion, Proposition \ref{prop: abstract pressure gap}. The proof draws inspiration from \cite[Theorem B]{burns2018unique}. However, our specific setup allows us to circumvent several technicalities and arrive at a more straightforward proof than that presented in \cite{burns2018unique}. In the appendix, we show that for warped product metrics, radial curvature bounds and Ricci curvature bounds are equivalent, and \textcolor{red}{we} provide a proof of an abstract result in ergodic theory, inspired by Peres' work, which we refer to as Peres' lemma for flows.

\subsection*{Acknowledgments}

\noindent The authors are grateful to Amie Wilkinson for proposing this question, to Jairo Bochi for bringing Peres' lemma to our attention,  and to Keith Burns, Vaughn Climenhaga, Todd Fisher, and Dan Thompson for enlightening discussions. We also thank the referees for their helpful comments in improving this work. Lastly, the authors dedicate this work to Todd Fisher for the inspiration he brought us in his too-short but luminous life.

\section{Preliminary\label{sec:Preliminary}}

\subsection{Geometry of nonpositively curved manifolds\label{subsec:Geometry-of-nonpositively}}

This subsection will survey relevant geometric features of
nonpositively curved manifolds. A more comprehensive survey of these
results can be found in \cite{ballmann1995lectures,eberlein2001geodesic}.

Let $M$ be a closed nonpositively curved manifold, and $\{f_{t}\}_{t\in\R}$
the geodesic flow on the unit tangent bundle $T^{1}M$. The tangent
bundle $TT^{1}M$ contains three $df_{t}$-invariant continuous bundles
$E^{s},E^{u}$, and $E^{c}$. The bundle $E^{c}$ is one-dimensional
along the flow direction, and the other two bundles $E^{s/u}$, which
are orthogonal to $E^{c}$ with respect to the Sasaki metric, can
be described using Jacobi fields. If $M$ is negatively curved,
these three bundles form a splitting of $TT^{1}M$.

A \textit{Jacobi field} $J$ along a geodesic $\gamma$ is a vector
field along $\gamma$ satisfying the \textit{Jacobi equation} 
\[
J''(t)+R(J(t),\gamma'(t))\gamma'(t)=0
\]
where $R$ is the Riemannian curvature tensor. A Jacobi field $J$ is
\textit{orthogonal} if there exists $t_{0}\in\R$ such that $J(t_{0})$
and $J'(t_0)$ are perpendicular to $\gamma'(t_{0})$. It is well known
that when this orthogonal property holds at some $t_{0}\in\R$, then
it holds for all $t\in\R$. A Jacobi field $J$ is \textit{parallel}
if $J'(t)=0$ for all $t\in\R$.

Denoting the space of orthogonal Jacobi fields along $\gamma$ by
$\J^{\perp}(\gamma)$, we can identify $T_{v}T^{1}M$ with $\J^{\perp}(\gamma_{v})$
as follows. Consider a vector $v\in T_{p}M$. Using the Levi-Civita
connection the tangent space $T_{v}TM$ at $v$ may be identified
with the direct sum $H_{v}\oplus V_{v}$ of horizontal and vertical
subspace, respectively. The \textit{Sasaki metric} on \( TM \) is defined by declaring \( H_v \) and \( V_v \) to be orthogonal, with both \( H_v \) and \( V_v \) being isomorphic to \( T_pM \) and equipped with the norm induced by the Riemannian metric on \( M \). Restricted to $T^{1}M$, the tangent space $T_{v}T^{1}M$
corresponds to $H_{v}\oplus v^{\perp}$ under this identification.
Then any vector $\xi\in T_{v}T^{1}M$ for $v\in T^{1}M$ may be written
as $(\xi_{h},\xi_{v})$ and can be identified with an orthogonal Jacobi
field $J_{\xi}\in\J^{\perp}(\gamma_{v})$ along $\gamma_{v}$ with
the initial conditions $(J_{\xi}(0),J'_{\xi}(0))=(\xi_{h},\xi_{v})$.
Moreover, the Sasaki norm of $df_{t}(\xi)$ satisfies 
\[
\|df_{t}(\xi)\|^{2}=\|J_{\xi}(t)\|^{2}+\|J'_{\xi}(t)\|^{2}.
\]

The stable subspace $E^{s}(v)$ at $v\in T^{1}M$ is then defined
as 
\[
E^{s}(v):=\{\xi\in T_{v}T^{1}M\colon\|J_{\xi}(t)\|\text{ is bounded for }t\geq0\}.
\]
Similarly, the unstable subspace $E^{u}(v)$ consists of vectors $\xi\in T_{v}T^{1}M$
where $\|J_{\xi}(t)\|$ is bounded for $t\leq0$. The subbundles $E^{u}(v)$ and $E^{s}(v)$ are integrable to the respective foliations $W^{u}(v)\subset T^{1}M$ and $W^{s}(v)\subset T^{1}M$.
The footprints of $W^{u}(v)$ and $W^{s}(v)$ on $M$ are called the
\textit{unstable }and \textit{stable} \textit{horospheres,} which
are denoted by $H^{u}(v)$ and $H^{s}(v)$, respectively.

The \textit{rank} of a vector $v\in T^{1}M$ is the dimension of the
space of parallel Jacobi fields along $\gamma_{v}$, which coincides
with the number $1+\dim(E^{s}\cap E^{u})$. We say the manifold is
\textit{rank 1} if it has at least one rank 1 vector. This paper focuses
mainly on closed rank 1 nonpositively curved manifolds.

The \textit{singular set} is a set of vectors on which the geodesic
flow fails to display uniform hyperbolicity, and it is defined by 
\[
\Sing:=\{v\in T^{1}M\colon E^{s}(v)\cap E^{u}(v)\neq0\}.
\]
The singular set is closed and $\F$-invariant, and in the case where
$M$ is a surface the singular set can be characterized as the set
of vectors $v$ where the Gaussian curvature $K(\gamma_{v}(t))$ vanishes
for all $t\in\R$. The complement of the singular set is the regular
set 
\[
\Reg:=T^{1}M\setminus\Sing.
\]

The geodesic flow restricted to the regular set is (nonuniformly) hyperbolic and
the degree of hyperbolicity can be measured by the function
\[
\lambda(v):=\min(\lambda^{u}(v),\lambda^{s}(v))
\]
where $\lambda^{u}(v)$ is the minimum eigenvalue of the shape operator
on the unstable horosphere $H^{u}(v)$ at $v$. Using $\lambda$ we
can define nested compact subsets $\{\Reg(\eta)\}_{\eta>0}$ of $\Reg$
where 
\[
\Reg(\eta):=\{v\in\Reg\colon\lambda(v)\geq\eta\}.
\]
These subsets may be viewed as uniformity blocks in the sense of Pesin's
theory, where the hyperbolicity is uniform. More details and properties
of the function $\lambda$ can be found in \cite{burns2018unique}.

The \textit{geometric potential} $\varphi^{u}:T^{1}M\to\R$ is an important potential that measures the infinitesimal volume growth
in the unstable direction: 
\[
\varphi^{u}(v):=-\lim_{t\to0}\frac{1}{t}\log\det(df_{t}\mid_{E^{u}(v)})=-\frac{d}{dt}\Big|_{t=0}\log\det(df_{t}\mid_{E^{u}(v)}).
\]

In order to study the geometric potential, it is convenient to study Riccati equations. Interestingly, the shape operator of unstable (and stable)
horosphere is a solution of a Riccati equation. To see this, we start
by introducing terminologies.

Let $H\subset M$ be a hypersurface orthogonal to $\gamma_{v}$ at
$\pi v$ where $\pi:T^{1}M\to M$ is the canonical projection. An
orthogonal Jacobi field $J\in\J^{\perp}(\gamma_{v})$ is called a
$H$-\textit{Jacobi field} along $\gamma_{v}$, if $J$ comes from
varying $\gamma_{v}$ through unit speed geodesics orthogonal to $H$.
We denote the set of $H$- Jacobi fields by $\J_{H}(\gamma_{v})$
. The \textit{shape operator} on $H$ is the symmetric linear operator
$U:T_{\pi v}H\to T_{\pi v}H$ defined by $U(w)=\nabla_{w}N$, where
$N$ is the unit normal vector field toward the same side as $v$.

We are particularly interested in the unstable horosphere $H=H^{u}(v)$
at $v$. In this case, $J_{H^{u}}(\gamma_{v})$ coincides with the
space of unstable Jacobi fields $J^{u}(\gamma_{v})$. For $t\in\R$,
let $U_{v}^{u}(t):T_{\pi f_{t}v}H^{u}(f_{t}v)\to T_{\pi f_{t}v}H^{u}(f_{t}v)$
be the shape operator of the unstable horosphere $H^{u}(f_{t}v)$.
We know $U_{v}^{u}(t)$ is a positive semidefinite symmetric linear
operator on $(f_{t}v)^{\perp}$, and for any unstable Jacobi field
$J(t)$ it satisfies $J'(t)=U_{v}^{u}(t)J(t)$; see \cite[Lemma 2.9]{burns2018unique}.

For any vector $v\in T^{1}M$, let $\mathcal{K}(v):v^{\perp}\to v^{\perp}$
is the symmetric linear map defined via $\langle\mathcal{K}(v)X,Y\rangle:=\langle R(X,v)v,Y\rangle$
for $X,Y\in v^{\perp}$. Using the Jacobi equation, for an unstable
Jacobi field $J(t)$ we know $J''(t)+K(f_{t}v)J(t)=0$ and $J'(t)=U_{v}^{u}(t)J(t)$.
Thus, we get the operator-valued Riccati equation: 
\begin{equation}
(U_{v}^{u})'(t)+U_{v}^{u}(t)^{2}+{\cal K}(\dot{\gamma}_{v}(t))=0;\label{eq:operator-Riccati}
\end{equation}
see \cite[(7.6)]{burns2018unique}. Using the above notation, the
Ricci curvature ${\rm Ric}(v)$ at $v$ is defined as the trace of
the map ${\cal K}(v)$.

\subsection{Thermodynamic formalism\label{subsec:Thermodynamic-formalism}}

We now briefly survey relevant results in thermodynamic formalism.
The general notion of topological entropy and pressure described in
the following can be defined for an arbitrary flow $\F=\{f_{t}\}_{t\in\R}$
in a compact metric space $(X,d)$.

For any $t>0$, we define a metric $d_{t}$ on $X$ via 
\[
d_{t}(x,y):=\max\limits _{0\leq\tau\leq t}d(f_{\tau}x,f_{\tau}y),
\]
and the corresponding $\delta$-ball around $x\in X$ in $d_{t}$-metric
will be denoted by $B_{t}(x,\delta)$. We say a subset $E$ of $X$
is \textit{$(t,\delta)$-separated} if $d_{t}(x,y)\geq\delta$ for
distinct $x,y\in E$. Moreover, we will identify $(x,t)\in X\times[0,\infty)$
with the orbit segment of length $t$ starting at $x$.

Let $\vp\colon X\to\R$ be a continuous function on $X$, which we
often call a \textit{potential}. We define ${\displaystyle \Phi(x,t):=\int_{0}^{t}\vp(f_{\tau}x)\,d\tau}$
to be the integral of $\vp$ along an orbit segment $(x,t)$. For
any subset $\CC\in X\times[0,\infty)$, we let $\CC_{t}$ be the subset
of $\CC$ consisting of orbit segments of length $t$. We define 
\[
\Lambda(\CC,\vp,\delta,t)=\sup\Big\{\sum\limits _{x\in E}e^{\Phi(x,t)}\colon E\subset\CC_{t}\text{ is }(t,\delta)-\text{separated}\Big\}.
\]

The \textit{topological pressure} of $\vp$ on $\CC$ is then defined
by 
\[
P(\CC,\vp):=\lim\limits _{\delta\to0}\limsup\limits _{t\to\infty}\frac{1}{t}\log\Lambda(\CC,\vp,\delta,t).
\]
When $\CC$ is the entire orbit space $X\times[0,\infty)$, then we
denote it by $P(\vp)$ and call it the \textit{topological pressure}
of $\vp$. In the case where $\vp\equiv0$, the resulting pressure
$P(0)$ is called the \textit{topological entropy} of the flow $\F$
denoted by $h_{top}(\F)$.

Denoting by $\mathcal{M}(\F)$ the set of all $\F$-invariant measures
on $X$, the pressure $P(\vp)$ satisfies the \textit{variational
principle} 
\[
P(\vp)=\sup\limits _{\mu\in\mathcal{M}(\F)}\Big\{ h_{\mu}(\F)+\int\vp\,d\mu\Big\}
\]
where $h_{\mu}(\F)$ is the measure-theoretic entropy of $\mu$. Any
invariant measure $\mu\in\mathcal{M}(\F)$ attaining the supremum
is called an \textit{equilibrium state} for $\vp$. Likewise, any
invariant measure attaining the supremum when $\vp\equiv0$ is called
a \textit{measure of maximal entropy}.

\subsection{Codimension 1 totally geodesic flat torus and Fermi coordinates}

\label{sec:codim1}

Let $M$ be an $n$-dimensional closed rank 1 nonpositively curved
manifold and $T_{0}$ a totally geodesic $(n-1)$-torus in $M$ with
$K\equiv0$ on any $x\in T_{0}$. We further suppose that the complement
of $T_{0}$ is negatively curved and that curvature away from a small
neighborhood of $T_{0}$ admits a uniform upper bound strictly smaller
than 0. A more precise control of the curvature of the neighborhood
will be specified later, depending on the setting under consideration.

In what follows, we fix a fundamental domain in $\widetilde{M}$ the
universal covering of $M$ and (abusing the notation) continue denoting
the lifts of $p\in M$ and $v\in T_{p}M$ by $p$ and $v$, respectively.
Recall that the Fermi coordinate of $p$ is given by $(s,x)$ where
$s$ is an $(n-1)$-dimensional coordinate on $\widetilde{T}_{0}$
and $x$ measures the signed distance on $\widetilde{M}$ to $\widetilde{T}_{0}$.

For $p\in\widetilde{M}$ near $\widetilde{T}_{0}$, by $x(p)$ we
mean the $x$-coordinate of $p$. For any $v\in T_{p}\widetilde{M}$
with $p$ near $\widetilde{T}_{0}$, we define $x_{v}:=x(p)$ and
denote by $\phi_{v}$ the signed angle between $v$ and the hypersurface
$x=x(p)$; we adopt the convention that $\phi_{X}=\pi/2$ when $X=\partial/\partial x$ is the vertical vector field.
We also define 
\begin{equation}
x_{v}(\tau):=x(\gamma_{v}(\tau))\text{ and }\phi_{v}(\tau):=\phi(\gamma_{v}'(\tau)).\label{eq: notation}
\end{equation}
When there is no confusion, we may write $x_{v}$ and $\phi_{v}$
for $x_{v}(0)$ and $\phi_{v}(0)$, respectively.
\begin{rem}
\label{rem: Fermi coords} With respect to Fermi coordinates $(s,x)$,
the curve $s=\text{const.}$ is always a geodesic perpendicular to
$\widetilde{T}_{0}$, while $(s(t),x(t))$ with $x(t)\equiv x_{0}$
for some $x_{0}$ is not a geodesic unless $x_{0}=0$ and $s(t)$
is linear. 
\end{rem}

For $\ep>0$ small, the Riemannian metric near $\widetilde{T}_{0}$
can be written as 
\begin{equation}
g=dx^{2}+g_{x},\quad|x|\leq\ep\label{eqn: metric_HD}
\end{equation}
where $g_{x}$ is the Riemannian metric on $\widetilde{T}_{x}:=\widetilde{T}_{0}\times\{x\}$.
In particular, $g_{0}$ is the Euclidean metric on $\widetilde{T}_{0}$.

Recall that the \textit{second fundamental form on $\widetilde{T}_{x}$} is defined
via 
\[
\II(v,w):=\langle\nabla_{v}X,w\rangle,
\]
for any $v,w\in T_{(s,x)}\widetilde{T}_{x}$. The \textit{shape operator
$U(s,x):T_{(s,x)}\widetilde{T}_{x}\to T_{(s,x)}\widetilde{T}_{x}$}
is defined via 
\[
\II(v,w)=\langle U(s,x)v,w\rangle.
\]
As $\II$ is bilinear and symmetric, the shape operator $U(s,x)$
is diagonalizable. Its eigenvalues 
\[
\lambda_{1}(s,x)\leq\lambda_{2}(s,x)\leq\cdots\leq\lambda_{n-1}(s,x)
\]
are called \textit{principal curvatures at $(s,x)$}. When $s=0$, since $T_0$ is totally geodesic, we have $\lambda_{1}(0,x)=\lambda_{2}(0,x)=\cdots=\lambda_{n-1}(0,x)=0$.
Moreover, $U(s,x)$ satisfies the following Riccati equation:
\[U_s+U^2+K=0.\]
Since $K\leq 0$, we know that $U$ is positive semi-definite, thus there exists $R>0$ such that 
\begin{equation}
\lambda_{1}(s,x)\geq 0,\quad \text{ for } s\in [0,R]\label{eqn: positive_lambda}
\end{equation}
For any geodesic $\gamma(t)=(s(t),x(t))$ near $\widetilde{T}_{0}$,
by the first variation formula, we have 
\begin{equation}
x'=\sin\phi,\label{eqn: phi_eqn_HD}
\end{equation}
where $\phi(t):=\phi(\gamma'(t))$, which then gives $x''=\phi'\cos\phi$.

We denote by 
\[
\gamma'_{\perp}(t):=\gamma'(t)-\langle\gamma'(t),X\rangle X
\]
the component of $\gamma'(t)$ that is orthogonal to $X$. Then $|\gamma'_{\perp}|=\cos\phi$. 
\begin{lem}
\label{lem: geod_eqn_HD} If $\gamma(t)=(s(t),x(t))$ is a geodesic
on $\widetilde{M}$ near $\widetilde{T}_{0}$, we have 
\[
x''=\II(\gamma'_{\perp},\gamma'_{\perp})\in[\lambda_{1}(s,x)\cos^{2}\phi,\lambda_{n-1}(s,x)\cos^{2}\phi].
\]
\end{lem}

\begin{proof}
By \eqref{eqn: phi_eqn_HD}, we have $x'=\sin\phi=\langle X,\gamma'(t)\rangle$.
Thus 
\[
x''=\frac{d}{dt}\langle X,\gamma'(t)\rangle=\langle\nabla_{\gamma'(t)}X,\gamma'(t)\rangle=\langle\nabla_{\gamma'_{\perp}(t)}X,\gamma'_{\perp}(t)\rangle=\II(\gamma'_{\perp},\gamma'_{\perp}).
\]
Note that the third equality uses the fact $\nabla_{X}X=0$; see Remark
\ref{rem: Fermi coords}. Since $|\gamma'_{\perp}|=\cos\phi$, we
have 
\[
\II(\gamma'_{\perp},\gamma'_{\perp})\geq\lambda_{1}(s,x)|\gamma'_{\perp}|^{2}=\lambda_{1}(s,x)\cos^{2}\phi
\]
and 
\[
\II(\gamma'_{\perp},\gamma'_{\perp})\leq\lambda_{n-1}(s,x)\cos^{2}\phi.
\]
This completes the proof. 
\end{proof}
\begin{rem}
\label{rem: surface_metric} When $M$ is a surface, then $T_{0}$
is a closed geodesic. In this case the Riemannian metric \eqref{eqn: metric_HD}
near $\widetilde{T}_{0}$ may be written as 
\begin{equation}
g=dx^{2}+G(s,x)^{2}ds^{2}\label{eqn: metric}
\end{equation}
with $G(s,0)\equiv1$. The Gaussian curvature is given by $K=-G_{xx}/G$,
and the second fundamental form at $\widetilde{T}_{x}$ is $G_{x}/G$.
In particular, by Lemma \ref{lem: geod_eqn_HD}, $x''$ admits the
following expression 
\begin{equation}
x''=\lambda(s,x)\cos^{2}\phi=\frac{G_{x}}{G}\cos^{2}\phi.\label{eqn: geod_eqn}
\end{equation}
\end{rem}

\section{Preparation and outline for pressure gaps results\label{sec:Preparation}}

\subsection{Shadowing map}

\label{subsec:Shadowing-map}

To distinguish vectors in $\Sing$ and generic vectors in $T^{1}M$,
we will use different fonts to denote them; more precisely, we will
write $\v\in\Sing$ and $v\in T^{1}M$. Given an orbit segment $(\v,t)\in T^{1}M\times\R^{+}$
with $\v\in\Sing$, we now describe a method for constructing a new
orbit segment that shadows $(\v,t)$. Though simple, this construction
will be crucial in proving Theorem \ref{Main-result} and Theorem
\ref{Main-result-3}.

For any $\v\in\text{Sing}$, suppose $\pi(\v)=(s_{0},0)$ for
some $s_{0}$. For any $t>0$ and any $R>0$ such that the Fermi coordinates
$(s,x)$ are well-defined for $|x|<R$, there exists $s_{1}\in\R$
such that the distance on $\widetilde{M}$ between $(s_{0},R)$ and
$((s\circ \pi)(f_{s_{1}}\v),R)$ is equal to $t$; see Figure \ref{fig1}.  From the triangle inequality we know that $|t-s_1|<2R$.

Denoting by $\gamma$ the geodesic connecting these two points, we define 
\begin{equation}
\Pi_{t,R}(\v):=\gamma'(0).\label{eq: Pi}
\end{equation}
Throughout the paper, we will often write $w_{\v}$, or simply $w$,
to denote $\Pi_{t,R}(\v)$ whenever the context is clear. 

\begin{figure}[h]
\centering

\includegraphics[width=12cm]{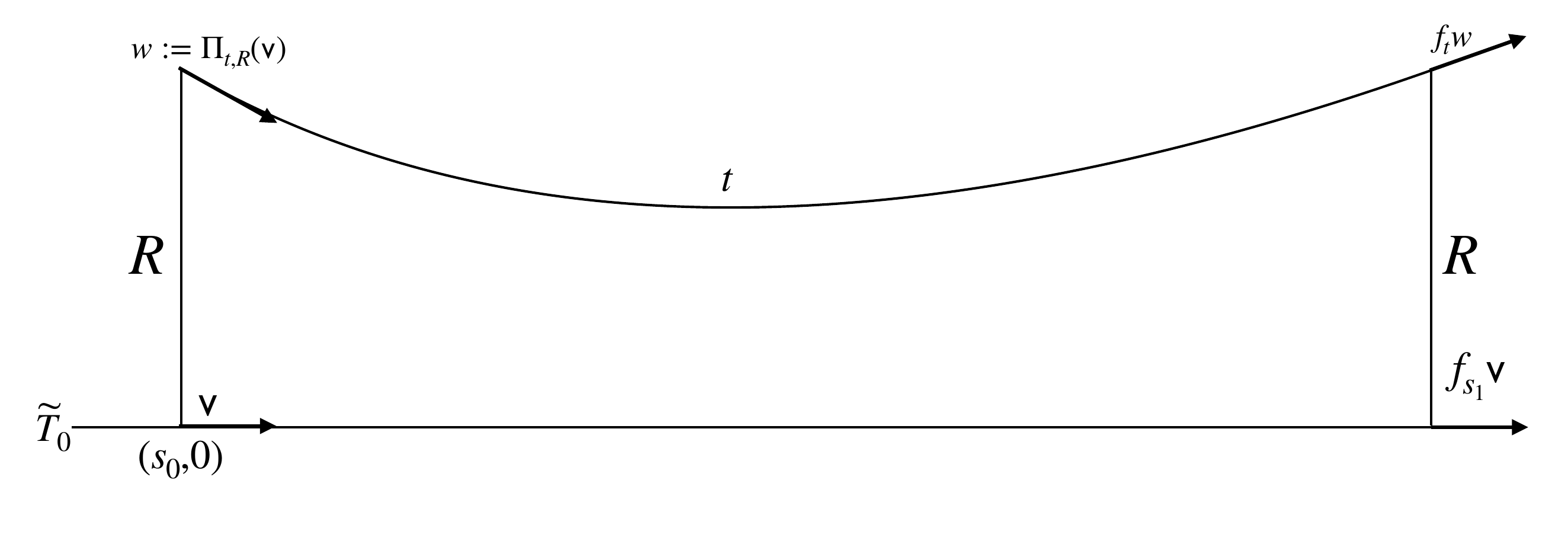} \caption{Construction of $\Pi_{t,R}$}
\label{fig1} 
\end{figure}

The next few lemmas establish a few properties on the map $\Pi_{t,R}$. 
\begin{lem}
\label{lem: angle_lower_bdd_reg} There exists $R_{0}>0$ such that for any $R\in(0,R_{0})$,  any $\v\in\Sing$ and any $t>0$,
the following statements hold: 

\begin{enumerate}[font=\normalfont]
\item The function $\tau\mapsto x_{w}(\tau)$ is convex and positive for any $\tau\in[0,t]$. 
\item $|\phi_{w}(\tau)|<\pi/4$
for all $\tau\in[0,t]$. 
\item For any $R$, there exists $\eta=\eta(R)>0$ such that $w,f_{t}w\in\text{Reg}(\eta)$. 
\end{enumerate}
\end{lem}

\begin{proof}
The convexity in (1) is a consequence of \eqref{eqn: positive_lambda} and Lemma \ref{lem: geod_eqn_HD}. Now we prove positivity. If $x_{v}$ vanishes somewhere in $(0,t)$, assume $\wt{a}$ is
the smallest zero of $x_{v}$ in $(0,t)$. It is clear that $x_{v}'(\wt{a})\leq0$.
If $x_{v}'(\wt{a})=0$, then it would imply $\gamma_{v}\subseteq\widetilde{T}_{0}$
which is impossible. Thus $x_{v}'(\wt{a})<0$, and denote by $\wt{b}$
the next zero of $x_{v}$. Then $x_{v}$ is a geodesic connecting
two distinct points on a totally geodesic submanifold $\widetilde{T}_{0}$,
which would imply $\gamma_{v}\subseteq\widetilde{T}_{0}$, again resulting
in a contradiction. Therefore, $x_{v}(\tau)$ is positive for all
$\tau\in[0,t]$.  

For (2), it is clear that 
\[
|\phi_{w}(\tau)|\leq\max\{|\phi_{w}(0)|,|\phi_{w}(t)|\}
\]
for all $\tau\in[0,t]$. We then observe that $|\phi_{w}(0)|\leq|\phi(\gamma'_{w,s}(0))|$
where $\gamma_{w,s}$ is the geodesic with the same initial point
as $\gamma_{w}$ that is forward asymptotic to $\widetilde{T}_{0}$.
The statement then follows as the function 
\[
R\mapsto\sup\{|\phi(u)|\colon u\in W^{s}(\v)\text{ for some }\v\in\Sing\text{ and }x(\pi u)=R\}
\]
vanishes when $R=0$ and varies continuously in $R$. By repeating
the same argument with the unstable manifolds $W^{u}(\v)$ for $\v\in\Sing$,
we can find $R_{0}>0$ with the desired property.

For (3), any unit vector $v\in T^{1}M$ that
satisfies $x(\pi(v))\neq0$ belongs to $\Reg$. Since $\Reg$
is exhausted by compact subsets $\Reg(\eta)$ and the flat torus $T_{0}$
in $M$ is compact, the statement follows. 
\end{proof}

\begin{rem}
\label{rem: phi} From here on, we will assume that $R$ belongs to
$[0,\min(R_{0},R_{1})]$ with $R_{0},R_{1}$ defined as in above lemmas.
In particular, we will often evoke Lemma \ref{lem: angle_lower_bdd_reg}
to use the inequality 
\[
\cos\phi_{w}(\tau)\in\Big[\frac{\sqrt{2}}{2},1\Big]
\]
for any $w=\Pi_{t,R}(\v)$ and $\tau\in[0,t]$. 
\end{rem}

\subsection{Key inequality and the outline of Theorem \ref{Main-result}}

\label{subsec: outline} In this subsection, we provide a brief outline
of what consists of the remaining sections. Recall from Lemma \ref{lem: angle_lower_bdd_reg}
that $x_{w}(\tau)$ is a convex function on $\tau\in[0,t]$ for any
$w=\Pi_{t,R}(\v)$, and hence there exists a well-defined number $\wt{t}\in[0,t]$
such that 
\begin{equation}
x_{w}(\wt{t})=\min_{\tau\in[0,t]}x_{w}(\tau)\label{eq: tmin}
\end{equation}
and that $\wt{t}$ is the smallest among all such numbers. In the
case where $x_{w}(\tau)$ is strictly convex, there is a unique $\wt{t}$
which attains the minimum of $x_{w}(\tau)$. Note that $x'_{w}(\tau)\leq0$
and $\phi_{w}(\tau)\leq0$ for $\tau\in[0,\wt{t}\,]$.

The goal of the next few sections is to establish bounds on the distance
$x_{w}(\tau)$ and the angle $\phi_{w}(\tau)$ under the assumption
that the radial curvature $K_{\perp}$ vanishes to the order of $m-1$
at $T_{0}$ and that the curvature near $T_{0}$ is controlled; see
Section \ref{sec: higher dim} and \ref{sec: GW surface} for the
precise description of the setting. In particular, we will show that
for suitable $R>0$, there exists $Q=Q(R)>1$ independent of $t$
such that the shadowing vector $w=\Pi_{t,R}(\v)$ for any $\v\in\Sing$
and any $t>0$ satisfies 
\begin{align}
\begin{split}Q^{-1}(\tau+1)^{-\frac{2}{m}}\leq & x_{w}(\tau)\leq Q(\tau+1)^{-\frac{2}{m}},\\
 & |\phi_{w}(\tau)|\leq Q[(\tau+1)^{-\frac{m+2}{m}}-(\wt{t}+1)^{-\frac{m+2}{m}}]
\end{split}
\label{eq: main}
\end{align}
for any $\tau\in[0,\wt{t}\,]$. From its derivation in Proposition
\ref{prop: estm_x_psi} and \ref{prop: estm_x_psi_GW}, it will be
clear that the analogous inequality holds for $\tau\in[\wt{t},t]$
by simply applying the symmetric argument starting from $\tau=t$
instead of $\tau=0$: 
\begin{align}
\begin{split}Q^{-1}(t-\tau+1)^{-\frac{2}{m}}\leq & x_{w}(\tau)\leq Q(t-\tau+1)^{-\frac{2}{m}},\\
 & |\phi_{w}(\tau)|\leq Q[(t-\tau+1)^{-\frac{m+2}{m}}-(\wt{t}+1)^{-\frac{m+2}{m}}].
\end{split}
\label{eq: main_sub}
\end{align}
In the setting considered in Section \ref{sec: higher dim} where
a uniform control on the curvature $K_{\perp}$ is assumed on the
entire $\widetilde{T}_{0}$, the angle $|\phi_{w}(\tau)|$ from \eqref{eq: main}
and \eqref{eq: main_sub} admits the corresponding lower bound also;
see Proposition \ref{prop: estm_x_psi}.

The examples analyzed by Gerber and Wilkinson \cite{gerber1999holder} are indeed surfaces that satisfy the assumption on $K_{\perp}$ described above. Moreover, the potentials defined over these surfaces, which satisfy \eqref{eq:main-ineq}, are closely related to the geometric potential. Further details are provided in Section \ref{sec: GW surface}.

Assuming the estimates \eqref{eq: main} and \eqref{eq: main_sub},
we now derive useful consequences from them when considering potentials
$\vp$ satisfying \eqref{eq:main-ineq}. We will see in Section \ref{sec: pressure gap}
that, together with certain properties of the geodesic flow, these
results serve as sufficient criteria for the potential $\vp$ to have
the pressure gap. From direct integration using the estimates \eqref{eq: main}
and \eqref{eq: main_sub} we immediately get

\begin{prop}[Key inequality]
\label{prop: bowen_ball_int_esti} For some $R>0$, suppose that the
shadowing vector $w=\Pi_{t,R}(\v)$ satisfies \eqref{eq: main} and
\eqref{eq: main_sub} and $\varphi$ satisfies the assumption of Theorem \ref{Main-result}. Then there exists $Q=Q(R,C,\ep_{1},\ep_{2},\delta)>0$ such
that 
\[
\int_{0}^{t}\varphi(f_{\tau}u)d\tau\geq \int_0^{s_1}\varphi(f_{\tau}\v)d\tau-Q
\]
for any $u\in B_{t}(w,\d)$ where $w=\Pi_{t,R}(\v)$ for some $\v\in\Sing$. 
\end{prop}
\begin{proof}
For Type 1 manifolds, the proof is easier as $\varphi\equiv c$  on Sing.  By \eqref{eq:main-ineq},  \eqref{eq: main}
and \eqref{eq: main_sub},  we have
\[
\int_{0}^{t}\varphi(f_{\tau}u)d\tau\geq ct-Q'\geq cs_1-Q
\]
as $|t-s_1|\leq 2R.$\\
For Type 2 manifolds, we first prove the inequality for $u=w$, using \eqref{eq:main-ineq},  \eqref{eq: main}
and \eqref{eq: main_sub} again we have
\[
\int_{0}^{t}\varphi(f_{\tau}w)d\tau\geq \int_{0}^{t}\varphi(f_{\tau}ds(w))d\tau-Q= \int_0^{s_1}\varphi(f_{\tau}\v)d\tau -Q,
\]
where the equality follows the fact that Type 2 manifolds are surfaces.

For general $u\in B_{t}(w,\d)$,  
we begin by intersecting the geodesic $\gamma_{u}(\tau)$ with the
hyperplane $x=R$ so that there are exactly two intersection points,
each near $\pi(u)$ and $\pi(f_{t}u)$; see Figure \ref{fig2}. We
denote by $(u_{0},t_{0})$ the orbit segment connecting such intersection
points. Since two orbit segments $(u,t)$ and $(u_{0},t_{0})$ differ
only at either end by length at most $2\d$, there exists a constant
$C_{\d}$ depending only on $\d$ and $\|\vp\|$ such that 
\[
\Big|\int_{0}^{t}\vp(f_{\tau}u)d\tau-\int_{0}^{t_{0}}\vp(f_{\tau}u_{0})d\tau\Big|<C_{\d}.
\]
\begin{figure}[h]
\centering \includegraphics[width=12cm]{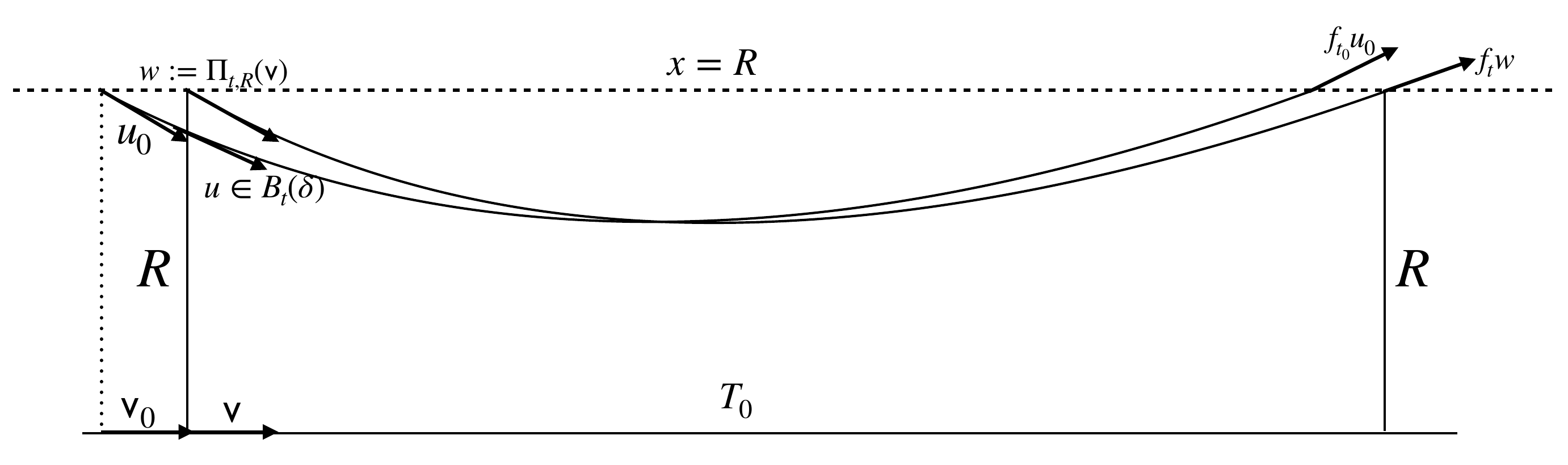} \caption{Figure for Proposition \ref{prop: bowen_ball_int_esti}}
\label{fig2} 
\end{figure}
Notice from its construction that $u_{0}$ is equal to $\Pi_{t_{0},R}(\v_{0})$
for some $\v_{0}\in\Sing$ near $v$, and hence the integral ${\displaystyle \int_{0}^{t_{0}}\vp(f_{\tau}u_{0})d\tau}$
admits a uniform lower bound independent of $t_{0}$ of the above
proposition. Therefore, the same is true for ${\displaystyle \int_{0}^{t}\vp(f_{\tau}u)d\tau}$
for $u\in B_{t}(w,\d)$. Lastly, since Type 2 manifolds are surfaces, the orbits of $\v_0$ and $\v$ largely overlap. As a result, the integrals along these two orbits remain uniformly bounded by a constant depending on $\delta$.

\end{proof}

\section{Estimates for Type 1 manifolds \label{sec: higher dim}}

Recall that $X=\partial/\partial x$ is the vertical vector field. 
For any $v\in T\widetilde{M}$ that is not collinear with $X$ we
define the \textit{radial curvature of $v$} by 
\begin{equation}
K_{\perp}(v):=K_{\sigma},\label{eq: normal curvature}
\end{equation}
that is, the sectional curvature of the plane $\sigma:=\text{span}\{v,X\}.$

In this section, we will consider the first of the two settings in
which $K_{\perp}$ vanishes uniformly to the order $m-1$. Namely,
if $(s,x)$ are the Fermi coordinates along $\widetilde{T}_{0}$,
there exists $C_{1},C_{2},\ep>0$ such that 
\begin{equation}
-C_{1}|x(v)|^{m}\leq K_{\perp}(v)\leq-C_{2}|x(v)|^{m}\label{eq: uniform curvature control}
\end{equation}
for any $v\perp X$ with $|x(v)|<\ep$.\\

As outlined in Subsection \ref{subsec: outline}, the main goal of
this section is to prove that under the above assumption on $K_{\perp}$,
the shadowing vector $w=\Pi_{t,R}(\v)$ for any $\v\in\Sing$ satisfies
the estimates on $x_{w}$ and $\phi_{w}$ claimed in \eqref{eq: main}.

Indeed, in this subsection, we will derive estimates on $x_{v}(t)$
for generic vectors $v\in T^{1}M$ near $\wt T_{0}$; namely, \textit{bouncing},
\textit{asymptotic}, and \textit{crossing vectors} (see Definition
\ref{def:bouncing_vec}). In Section \ref{sec:geo_potentials}, we
will discuss behaviors of geometric potentials $\vp^{u}$ with respect
to bouncing, asymptotic, and crossing vectors. Notice that by Definition
\ref{def:bouncing_vec}, shadowing vectors $w=\Pi_{t,R}(\v)$ are
bouncing vectors.

Let $N_{R}(\widetilde{T}_{0}):=\left\{ (s,x):s\in\widetilde{T}_{0},\ |x|<R\right\} $
be a neighborhood of $\widetilde{T}_{0}$, $v\in T^{1}M$ and $\gamma_{v}(t)=(s_{v}(t),x_{v}(t))$
in the Fermi coordinates. 
\begin{defn}
\label{def:bouncing_vec}Suppose $v\in T^{1}M$ such that $\gamma_{v}(0)\in N_{R}(\wt T_{0})$.
Let $-T_{1}=\inf\{t:\ \gamma_{v}(t)\in N_{R}(\widetilde{T}_{0})\}$
and $T_{2}=\sup\{t:\ \gamma_{v}(t)\in N_{R}(\widetilde{T}_{0})\}$.
We say that $v$ (relative to $N_{R}(\widetilde{T}_{0})$) is 
\begin{enumerate}
\item \textit{bouncing} if $T_{1},T_{2}<\infty$ and $x_{v}(t)>0$ or $x_{v}(t)<0$ for all $t\in(-T_{1},T_{2})$, 
\item \textit{asymptotic} if $T_{1}=\infty$ or $T_{2}=\infty$, 
\item \textit{crossing} if $T_{1},T_{2}<\infty$ and $x_{v}(t_{0})=0$ for some
$t_{0}\in(-T_{1},T_{2})$. 
\end{enumerate}
\end{defn}

Please see Fig \ref{fig:vec} for examples of these vectors. The definition
and study of these vectors are inspired by \cite{lima2021polydecay}. Later, in Lemma \ref{lem: grow_rate_x''_GW}, we show that the above definition is well-defined when \( R \) is sufficiently small. Moreover, by definition, all shadowing vectors are bouncing vectors and asymptotic
vectors are limiting cases of the bouncing vectors. More precisely,
one can regard asymptotic vectors as bouncing vectors $v$ with the
minimum of $x_{v}$ (when $\phi(v)<0$) occurring at $\wt t=\infty$
where, recalling from \eqref{eq: tmin}, $\wt{t}$ is the the time
(unique in this case) in which $x_{w}(t)$ attains its minimum.

\begin{figure}[h]
\centering \includegraphics[width=9cm]{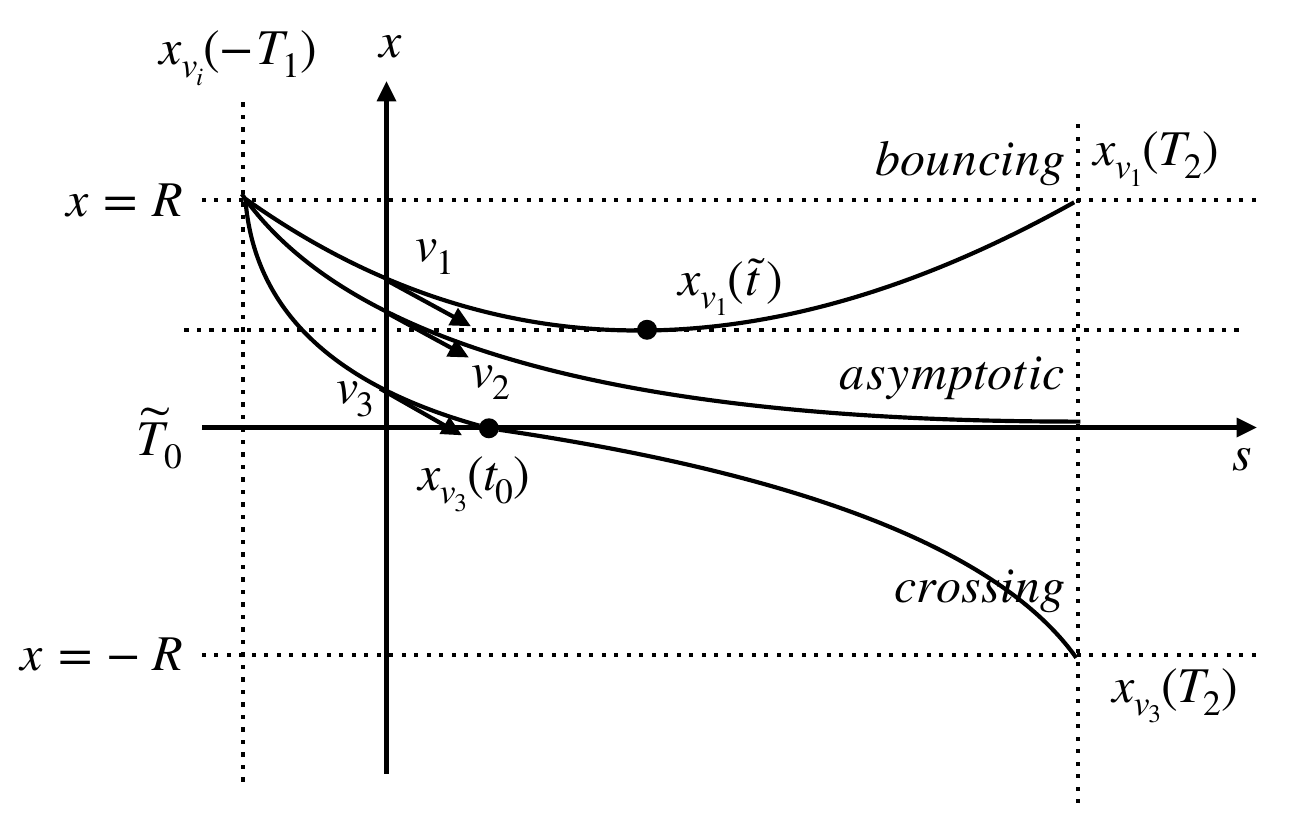} \caption{Bouncing, asymptotic and crossing vectors}
\label{fig:vec} 
\end{figure}

For Type 1 manifolds, the condition \eqref{eq: main} is established
in the proposition below (as for bouncing vectors):
\begin{prop}
\label{prop: estm_x_psi} For any $R>0$ sufficiently small (see Lemma
\ref{lem: grow_rate_x''_HD} for the domain $R$ can take), there
exists $Q=Q(R)>1$ independent of $t$ such that for any $v\in T^{1}M$
with $x_{v}(0)=R$ and $\phi_{v}(0)<0$, 

\begin{enumerate}[font=\normalfont]
\item if $v$ is a bouncing vector relative to $N_{R}(\widetilde{T}_{0})$,
then for $\tau\in[0,\wt{t}\,]$ we have 
\[
Q^{-1}(\tau+1)^{-\frac{2}{m}}\leq x_{v}(\tau)\leq Q(\tau+1)^{-\frac{2}{m}};
\]
\[
Q^{-1}[(\tau+1)^{-\frac{m+2}{m}}-(\wt{t}+1)^{-\frac{m+2}{m}}]\leq|\phi_{v}(\tau)|\leq Q[(\tau+1)^{-\frac{m+2}{m}}-(\wt{t}+1)^{-\frac{m+2}{m}}];
\]
\item if $v$ is an asymptotic vector relative to $N_{R}(\widetilde{T}_{0})$,
then for $\tau\in[0,\infty)$ we have 
\[
Q^{-1}(\tau+1)^{-\frac{2}{m}}\leq x_{v}(\tau)\leq Q(\tau+1)^{-\frac{2}{m}};
\]
\[
Q^{-1}[(\tau+1)^{-\frac{m+2}{m}}]\leq|\phi_{v}(\tau)|\leq Q[(\tau+1)^{-\frac{m+2}{m}}];
\]
\item if $v$ is a crossing vector relative to $N_{R}(\widetilde{T}_{0})$,
then for $\tau\in[0,t_{0}]$, we have 
\[
Q^{-1}[(\tau+1)^{-\frac{2}{m}}-(t_{0}+1)^{-\frac{2}{m}}]\leq x_{v}(\tau)\leq Q[(\tau+1)^{-\frac{2}{m}}-(t_{0}+1)^{-\frac{2}{m}}];
\]
\[
Q^{-1}(\tau+1)^{-\frac{m+2}{m}}\leq|\phi_{v}(\tau)|\leq Q(\tau+1)^{-\frac{m+2}{m}}.
\]
\end{enumerate}
\end{prop}

\begin{rem}
Proposition \ref{prop: estm_x_psi} is more general than \eqref{eq: main}.
The lower bound of $|\phi_{v}(\tau)|$ is not necessary for deriving 
Proposition \ref{prop: bowen_ball_int_esti}.
\end{rem}

We begin by collecting relevant lemmas to prove this proposition,
the first of which concerns the general property of Riccati solutions. 
\begin{lem}
\label{lem: ric_sol_comp} There exists $R_{0}=R_{0}(C,m)$ such that
for any $R\in[0,R_{0}]$, and any solution $U:[0,R]\to M_{n-1}(\mathbb{R})$
of the following Riccati equation 
\begin{equation}
U'+U^{2}-Cx^{m}I_{n-1}=0,\quad U(0)=0,\label{eqn: ric_eqn_comp}
\end{equation}
we have 
\[
\frac{Cx^{m+1}}{2(m+1)}I_{n-1}\leq U(x)\leq\frac{Cx^{m+1}}{m+1}I_{n-1}.
\]
\end{lem}

\begin{proof}
Let $\lambda:[0,R]\to\mathbb{R}$ be the solution of 
\[
\lambda'+\lambda^{2}-Cx^{m}=0,\quad\lambda(0)=0.
\]
Then $\lambda I_{n-1}$ satisfies \eqref{eqn: ric_eqn_comp}. Since
the solution of any first-order ODE is unique; we have $U=\lambda I_{n-1}$.

Now we estimate $\lambda$. Since $\lambda'=Cx^{m}-\lambda^{2}\leq Cx^{m}$,
we have 
\[
\lambda(x)=\int_{0}^{x}\lambda'(s)ds\leq\int_{0}^{x}Cs^{m}ds=\frac{C}{m+1}x^{m+1}.
\]
establishing the required upper bound on $U(x)$.

For the lower bound, let ${\displaystyle R_{0}:=\Big(\frac{(m+1)^{2}}{2C}\Big)^{\frac{1}{m+2}}}$.
Then for any $x\in[0,R_{0}]$ the upper bound for $\lambda(x)$ gives
\[
\lambda'=Cx^{m}-\lambda^{2}\geq Cx^{m}-\frac{C^{2}}{(m+1)^{2}}x^{2m+2}\geq\frac{C}{2}x^{m}
\]
which then gives ${\displaystyle \lambda(x)=\int_{0}^{x}\lambda'(s)ds\geq\frac{C}{2(m+1)}x^{m+1}}$
as required. 
\end{proof}
Recalling the notations $x_{v}(\tau):=x(\gamma_{v}(\tau))$ and $\phi_{v}(\tau):=\phi(\gamma_{v}'(\tau))$
from \eqref{eq: notation}, the following lemma uses the curvature
bound \eqref{eq: uniform curvature control} to compare $x''_{v}$
with $x_{v}^{m+1}$ for any shadowing vector $v$. 
\begin{lem}
\label{lem: grow_rate_x''_HD} There exists $R_{1}=R_{1}(C_{1},C_{2},m)>0$
such that for any $R\in[0,R_{1}]$ and any $v$ with $x_{v}(0)\leq R$,
we have 
\[
\frac{C_{2}}{4m+4}x_{v}^{m+1}\leq x_{v}''\leq\frac{C_{1}}{m+1}x_{v}^{m+1},
\]
as long as $0\leq x_v\leq R$. 
\end{lem}

\begin{proof}
Since $-C_{1}x(v)^{m}\leq K_{\perp}(v)\leq-C_{2}x(v)^{m},$ we have
\[
-C_{1}x^{m}I_{n-1}\leq\mathcal{K}(X)\leq-C_{2}x^{m}I_{n-1}
\]
\textcolor{black}{where $\mathcal{K}:v^{\perp}\to v^{\perp}$ is the
symmetric linear map such that $\langle\mathcal{K}(v)X,Y\rangle=\langle R(X,v)v,Y\rangle$
for $X,Y\in v^{\perp}$. } Using $R_{0}$ from the previous lemma,
we claim that we can take 
\[
R_{1}:=\min\{R_{0}(C_{1},m),R_{0}(C_{2},m)\}.
\]

For $i=1,2$ and $R\in(0,R_{1})$, let $U_{i}:[0,R]\to M_{n-1}(\mathbb{R})$
be the solution of 
\[
U_{i}'+U_{i}^{2}-C_{i}x^{m}I_{n-1}=0,\quad U_{i}(0)=0.
\]
By the main theorem in \cite{eschenburg1990comparison}, the solutions
satisfy 
\[
U_{2}\leq U\leq U_{1}
\]
on $[0,R]$. By Lemmas \ref{lem: geod_eqn_HD} and \ref{lem: ric_sol_comp}
and Remark \ref{rem: phi}, we have 
\[
x''_{v}(x)\geq\frac{\lambda_{1}(s,x_{v})}{2}\geq\frac{C_{2}}{4m+4}x_{v}^{m+1}
\]
and 
\[
x''_{v}(x)\leq\lambda_{n-1}(s,x_{v})\leq\frac{C_{1}}{m+1}x_{v}^{m+1}.
\]

\end{proof}

We also need the following auxiliary lemma. This lemma is inspired by standard comparison theorems in ordinary differential equations (ODEs).

\begin{lem}
\label{lem: estm_sol_ODE_discont} 

Let $f:[a,b]\to\R$ be a piecewise
smooth, strictly decreasing function with finitely many discontinuities.
Assume that $f(b)>0$ and that there exists $0<Q_{1}<Q_{2},\alpha>0,\beta\in(0,1)$
with $\alpha\beta>1$ such that 
\[
-Q_{2}(f(\tau)^{\alpha}-f(b)^{\alpha})^{\beta}\leq f'(\tau)\leq-Q_{1}(f(\tau)^{\alpha}-f(b)^{\alpha})^{\beta}
\]
when $f$ is smooth at $\tau$. Then there exists a constant $Q_{0}=Q_{0}(\alpha,\beta)>0$
independent of $a,b,f$ such that 
\[
(f(a)^{1-\alpha\beta}+Q_{2}(\alpha\beta-1)(\tau-a))^{\frac{1}{1-\alpha\beta}}\leq f(\tau)\leq Q_{0}(f(a)^{1-\alpha\beta}+Q_{1}(\alpha\beta-1)(\tau-a))^{\frac{1}{1-\alpha\beta}}
\]
for all $\tau\in[a,b]$. 
\end{lem}

\begin{proof}

We now briefly outline the steps of the proof. First, the lower bound presented in the lemma below can be obtained straightforwardly by integrating a suitable lower bound of $f'$. However, deriving the upper bound for $f$ is more intricate. Since a direct estimate from the bounds of $f'$ is not readily available, we construct an auxiliary ODE (\ref{eqn: ode_tilde_g_u}) that closely approximates the ODE providing a bound for the upper bound of $f'$, namely (\ref{eqn: ode_g_u}). We then estimate the solution $\widetilde g$ of (\ref{eqn: ode_tilde_g_u}) by integration. By applying a comparison argument, we establish a bound for $g$, which in turn implies the desired upper bound for $f$.

We first consider the lower bound of $f$. Firstly we have 
\[
\frac{d}{d\tau}f(\tau)^{1-\alpha\beta}=(\alpha\beta-1)(-f')f(\tau)^{-\alpha\beta}\leq Q_{2}(\alpha\beta-1),
\]
whenever $f$ is smooth. Thus, 
\begin{equation}
f(\tau)^{1-\alpha\beta}\leq f(a)^{1-\alpha\beta}+Q_{2}(\alpha\beta-1)(\tau-a).\label{eqn: upp_bdd_f_power}
\end{equation}
Hence, 
\[
f(\tau)\geq(f(a)^{1-\alpha\beta}+Q_{2}(\alpha\beta-1)(\tau-a))^{\frac{1}{1-\alpha\beta}}.
\]
Now, we compute the upper bound. Similar to the lower bound, we get
\[
\frac{d}{d\tau}f(\tau)^{1-\alpha\beta}=(\alpha\beta-1)(-f')f(\tau)^{-\alpha\beta}\geq Q_{1}(\alpha\beta-1)\left(1-\left(\frac{f(b)}{f(\tau)}\right)^{\alpha}\right)^{\beta},
\]
whenever $f$ is smooth.

We then define an auxiliary piecewise smooth function 
\[
g(\tau):=\left(\frac{f(\tau)}{f(b)}\right)^{1-\alpha\beta}
\]
which is strictly increasing on $[a,b]$ with $g(a)\in(0,1)$ and
$g(b)=1$. Moreover, 
\begin{equation}
\frac{dg}{d\tau}\geq Q_{1}(\alpha\beta-1)f(b)^{\alpha\beta-1}\left(1-g^{\frac{\alpha}{\alpha\beta-1}}\right)^{\beta}.\label{eqn: ode_g_u}
\end{equation}

Let $\wt{g}$ be the solution of the ODE 
\begin{equation}
\frac{d\tilde{g}}{d\tau}=Q_{1}(\alpha\beta-1)f(b)^{\alpha\beta-1}\left(1-\tilde{g}^{1/\beta}\right)^{\beta},\quad\quad\tilde{g}(a)=g(a).\label{eqn: ode_tilde_g_u}
\end{equation}

Since $\frac{\alpha}{\alpha\beta-1}>\frac{1}{\beta}$, from \eqref{eqn: ode_g_u}
and \eqref{eqn: ode_tilde_g_u} we know that $\tilde{g}'(a)<g'(a)$,
thus $\tilde{g}(\tau)<g(\tau)$ when $\tau$ is slightly larger than
$a$. In fact, we have $\tilde{g}(\tau)<g(\tau)$ for all $\tau\in(a,b)$.
This is because if $\tilde{g}(\tau)\geq g(\tau)$ for some $\tau$,
then we can define $\tau_{0}\in(a,b)$ to be the smallest $\tau$
with $\tilde{g}(\tau)=g(\tau)$. Since $\tilde{g}(\tau)<g(\tau)$
on $(a,\tau_{0})$, the condition $\tilde{g}(\tau_{0})=g(\tau_{0})$
implies $\tilde{g}'(\tau_{0})\geq g'(\tau_{0})$. On the other hand,
the condition $\tilde{g}(\tau_{0})=g(\tau_{0})$ considered with \eqref{eqn: ode_g_u}
and \eqref{eqn: ode_tilde_g_u} implies $\tilde{g}'(\tau_{0})<g'(\tau_{0})$,
deriving a contradiction. Thus $\wt{g}<g$ on $(a,b)$.

By \eqref{eqn: ode_tilde_g_u}, we have 
\[
\left(1-\tilde{g}^{1/\beta}\right)^{-\beta}\frac{d\tilde{g}}{d\tau}=Q_{1}(\alpha\beta-1)f(b)^{\alpha\beta-1}.
\]
Thus, for any $\tau\in[a,b]$, 
\begin{equation}
F_{\beta}(\tilde{g}(\tau))-F_{\beta}(\tilde{g}(a))=Q_{1}(\alpha\beta-1)f(b)^{\alpha\beta-1}(\tau-a),\label{eqn: F_b}
\end{equation}
where 
\[
F_{\beta}(x):=\int_{0}^{x}(1-y^{1/\beta})^{-\beta}\,dy.
\]
It is clear that $F_{\beta}$ is convex on $[0,1]$, thus $F'_{\beta}(x)\geq F'_{\beta}(0)=1$
for $x\in[0,1]$. Hence, for any $x\in[0,1]$, 
\begin{equation}
F_{\beta}(x)\geq x.\label{eqn: lower_bdd_F_b}
\end{equation}
On the other hand, since $F_{\beta}$ is convex and increasing, $F_{\beta}(x)/x$
is also increasing on $[0,1]$. Thus, for any $x\in(0,1)$ 
\begin{equation}
\frac{F_{\beta}(x)}{x}\leq F_{\beta}(1)=\int_{0}^{1}t^{\beta-1}(1-t)^{-\beta}\textcolor{red}{\beta} dt=\mathcal{B}(\beta,1-\beta)<\infty,\label{eqn: upper_bdd_F_b}
\end{equation}
where $\mathcal{B}$ is the beta function. By combining \eqref{eqn: F_b},
\eqref{eqn: lower_bdd_F_b}, and \eqref{eqn: upper_bdd_F_b}, we get
\[
\mathcal{B}(\beta,1-\beta)\tilde{g}(\tau)-\tilde{g}(a)\geq Q_{1}(\alpha\beta-1)f(b)^{\alpha\beta-1}(\tau-a).
\]
Hence 
\[
\mathcal{B}(\beta,1-\beta)f(\tau)^{1-\alpha\beta}-f(a)^{1-\alpha\beta}\geq Q_{1}(\alpha\beta-1)(\tau-a).
\]
Setting $Q_{0}:=\mathcal{B}(\beta,1-\beta)^{\frac{1}{\alpha\beta-1}}$,
we have 
\[
f(\tau)\leq Q_{0}(f(a)^{1-\alpha\beta}+Q_{1}(\alpha\beta-1)(\tau-a))^{\frac{1}{1-\alpha\beta}}.
\]
This completes the proof. 
\end{proof}
We are ready to prove Proposition \ref{prop: estm_x_psi}. 
\begin{proof}[Proof of Proposition \ref{prop: estm_x_psi}]
We will use $x$ to denote $x_{v}$ for simplicity. We will also
use $Q$ to denote a generic constant that may need to be updated;
this will be made clearer as they show up in the proof.

\textbf{Case 1: $v$ is a bouncing vector.} We will first prove the
lower bound for $x(\tau)$ when $\tau\in[0,\wt{t}\,]$. Noting that
$x'(\wt{t})=0$, by Lemma \ref{lem: grow_rate_x''_HD}, we have 
\[
x'(\tau)^{2}=-\int_{\tau}^{\wt{t}}2x'x''ds\leq\frac{2C_{1}}{m+1}\int_{\tau}^{\wt{t}}-x'x^{m+1}ds\leq\frac{2C_{1}}{(m+1)^{2}}(x(\tau)^{m+2}-x(\wt{t})^{m+2})
\]
and 
\[
x'(\tau)^{2}=-\int_{\tau}^{\wt{t}}2x'x''ds\geq\frac{C_{2}}{2m+2}\int_{\tau}^{\wt{t}}-x'x^{m+1}ds\geq\frac{C_{2}}{2(m+2)^{2}}(x(\tau)^{m+2}-x(\wt{t})^{m+2}).
\]
By taking $\alpha=m+2,\beta=1/2,f(a)=R$ in Lemma \ref{lem: estm_sol_ODE_discont},
we know that there exists $Q$ independent of $v,t$ such that for
any $\tau\in[0,\tilde{t}]$, 
\[
Q^{-1}(\tau+1)^{-\frac{2}{m}}\leq x(\tau)\leq Q(\tau+1)^{-\frac{2}{m}}.
\]

For $\phi$, recall that $x'=\sin\phi$. Thus $(\sin\phi)'=x''$.
By Lemma \ref{lem: grow_rate_x''_HD} we have 
\[
Q^{-1}(\tau+1)^{-\frac{2(m+1)}{m}}\leq\frac{C_{2}}{4m+4}x^{m+1}\leq(\sin\phi)'\leq\frac{C_{1}}{m+1}x^{m+1}\leq Q(\tau+1)^{-\frac{2(m+1)}{m}}.
\]
Taking the integral on $[\tau,\tilde{t}]$, we get 
\[
Q^{-1}[(\tau+1)^{-\frac{m+2}{m}}-(\tilde{t}+1)^{-\frac{m+2}{m}}]\leq\sin\phi(\tau)\leq Q[(\tau+1)^{-\frac{m+2}{m}}-(\wt{t}+1)^{-\frac{m+2}{m}}].
\]
Since $\sin\phi/\phi\in[2/\pi,1]$, we have the same bounds for $\phi$.

\textbf{Case 2: $v$ is an asymptotic vector.} It is not hard to see
that the above argument is valid when $\wt t=\infty.$

\textbf{Case 3: $v$ is a crossing vector.} Denote by $a(t):=-\sin\phi(t)$.
By Lemma \ref{lem: grow_rate_x''_HD}, we know that whenever $0\leq x\leq R$,
\[
\frac{C_{2}}{4m+4}x^{m+1}\leq x''\leq\frac{C_{1}}{m+1}x^{m+1}.
\]
Thus 
\[
\frac{C_{2}}{4m+4}(-x'x^{m+1})\leq-x'x''\leq\frac{C_{1}}{m+1}(-x'x^{m+1}).
\]
Taking integral, we get 
\begin{equation}
\frac{C_{2}}{(2m+2)^{2}}x^{m+2}\leq a^{2}-a(t_{0})^{2}\leq\frac{2C_{1}}{(m+1)^{2}}x^{m+2}.\label{eqn: bdd_x'_crossing-1}
\end{equation}
Hence 
\[
\frac{(m+1)^{2}}{2C_{1}}(a^{2}-a(t_{0})^{2})\leq x^{m+2}\leq\frac{(2m+2)^{2}}{C_{2}}(a^{2}-a(t_{0})^{2})
\]
Since $x'=-a$, we have $a'=-x''$. Thus there exists $C>c>0$ independent
of $v,t$ such that 
\[
-C(a^{2}-a(t_{0})^{2})^{\frac{m+1}{m+2}}\leq a'\leq-c(a^{2}-a(t_{0})^{2})^{\frac{m+1}{m+2}}.
\]
Setting $\alpha=2,\beta=\frac{m+1}{m+2}$ in Lemma \ref{lem: estm_sol_ODE_discont},
we have 
\begin{equation}
(a(0)^{-\frac{m}{m+2}}+C(\tau-a))^{-\frac{m+2}{m}}\leq a(\tau)\leq Q_{0}(a(0)^{-\frac{m}{m+2}}+c(\tau-a))^{-\frac{m+2}{m}}\label{eqn: bdd_a_crossing}
\end{equation}
Since $x_{v}(0)=R$ and $a(0)>0$, by compactness, $a(0)$ has a uniform
lower bound depending on $R$. Thus, we have 
\[
Q^{-1}(\tau+1)^{-\frac{m+2}{m}}\leq|\phi(\tau)|\leq Q(\tau+1)^{-\frac{m+2}{m}}.
\]
Similar to Case 1, the bound of $x$ comes from $x(\tau)=\int_{\tau}^{t_{0}}a(s)ds$
and \eqref{eqn: bdd_a_crossing}. 
\end{proof}
\begin{rem}
Note that we did not use the full strength of Lemma \ref{lem: estm_sol_ODE_discont}
in the above proof; that is, $x(\tau)$ had no discontinuities. The
next section will make similar use of Lemma \ref{lem: estm_sol_ODE_discont}
applied to a function with discontinuities. 
\end{rem}

\section{Estimates for Type 2 surfaces\label{sec: GW surface}}

In this section, we consider a different setting considered by Gerber
and Niţică \cite{gerber1999ETDS} as well as Gerber and Wilkinson
\cite{gerber1999holder} where $M=S$ is a complete nonpositively
curved surface, and $T_{0}$ is a closed geodesic of some length $\gamma_{0}$
on which the Gaussian curvature $K$ vanishes to order $m-1$. Namely,
if $(s,x)$ are the Fermi coordinates along $\widetilde{T}_{0}$,
there exists $C_{1},C_{2},\ep>0$ such that 
\begin{equation}
-C_{1}|x|^{m}\leq K(s,x)\leq0\label{eq: nonuniform-1}
\end{equation}
for all $|x|<\ep$ and for all $s\in\R$ and 
\begin{equation}
-C_{1}|x|^{m}\leq K(s,x)\leq-C_{2}|x|^{m}\label{eq: nonuniform-2}
\end{equation}
for all $|x|<\ep$ and  $s\in\wt{L}$ which is the lift of some interval $L\subset T_0$. To simplify the argument, whenever applicable, we will adopt the notation for the Riemannian metric $g$ specified for a surface introduced in Remark \ref{rem: surface_metric}. 
\begin{rem}
Compared to the assumption in the previous section, the underlying
manifold considered in this section is 2-dimensional, and the curvature
assumption near $T_{0}$ is weakened: the neighborhood of only a small
subset $L$ of $T_{0}$ is assumed to satisfy the uniform curvature
bound as in \eqref{eq: uniform curvature control}.

On the complement of $L$ in $T_{0}$ and its neighborhood, only the
trivial upper bound (i.e., zero) is imposed on the curvature. Despite
the weaker assumption on the curvature, the low dimensionality of
the manifold enables us to do a finer analysis to prove the similar estimates
\eqref{eq: main} on $x_{w}$ and $\phi_{w}$ for $w=\Pi_{t,R}(\v)$.
Furthermore, unlike in the previous section where $R$ from the definition
\eqref{eq: Pi} of the shadowing map $\Pi_{t,R}$ had to be carefully
chosen, this setting is less sensitive to the choice of $R$. 
\end{rem}

\begin{rem}
One can continue using techniques developed in the previous section
to study bouncing, asymptotic, and crossing vectors. However, the geometric potential estimates were well studied in the surface setting in \cite{gerber1999holder}. Without deviating from the main goal
and to simplify the argument in this section, we will only focus on
shadowing vectors $w=\Pi_{t,R}(\v)$. 
\end{rem}

The goal of this section is to show that under this different set
of assumptions, the shadowing vector $w=\Pi_{t,R}(\v)$ satisfies
the estimates in $x_{w}$ and $\phi_{w}$ as claimed in \eqref{eq: main}.
Recalling that $\gamma_{0}$ is the length of $T_{0}$, we state it
as a proposition below, which is the analog of Proposition \ref{prop: estm_x_psi}. 
\begin{prop}
\label{prop: estm_x_psi_GW} There exists $Q=Q(R)>1$ independent
of $t$ such that for any shadowing vector $w=\Pi_{t,R}(\v)$ and
$\tau\in[0,\wt{t}]$ we have 
\[
Q^{-1}(\tau+1)^{-\frac{2}{m}}\leq x_{w}(\tau)\leq Q(\tau+1)^{-\frac{2}{m}},
\]
and 
\[
|\phi_{w}(\tau)|\leq Q[(\tau+1)^{-\frac{m+2}{m}}-(\wt{t}+1)^{-\frac{m+2}{m}}],
\]
and for any $\tau\in[0,\wt{t}-2\sqrt{2}\gamma_{0}]$, 
\[
|\phi_{w}(\tau)|\geq Q^{-1}[(\tau+1)^{-\frac{m+2}{m}}-(\wt{t}+1)^{-\frac{m+2}{m}}].
\]
\end{prop}


To prove the proposition, we need to exploit the assumptions
on $K(s,x)$ and establish a few auxiliary lemmas. Consider any $w=\Pi_{t,R}(\v)$
for some $t,R>0$ and $\v\in\Sing$. Recall that $\wt{t}\in[0,t]$
is the smallest number in which $x_{w}(\tau)$ attains the minimum.
We can decompose $[0,\wt{t}]$ into subintervals $I_{i}=[t_{i},t'_{i}]$
and $I'_{i}=(t'_{i},t_{i+1})$ for $i=0,1,\cdots,n$ so that 
\[
s(\gamma_{w}(\tau))\in\wt{L}\text{ when }\tau\in I_{i}
\]
and 
\[
s(\gamma_{w}(\tau))\notin\wt{L}\text{ when }\tau\in I'_{i};
\]
see Figure \ref{fig3}. Notice that $\wt{t}\geq t_{n}$, thus $I_{n}$
is not empty (though it may be arbitrarily short), but $I'_{n}$ may
be empty.

\begin{figure}[h]
\centering \includegraphics[width=12cm]{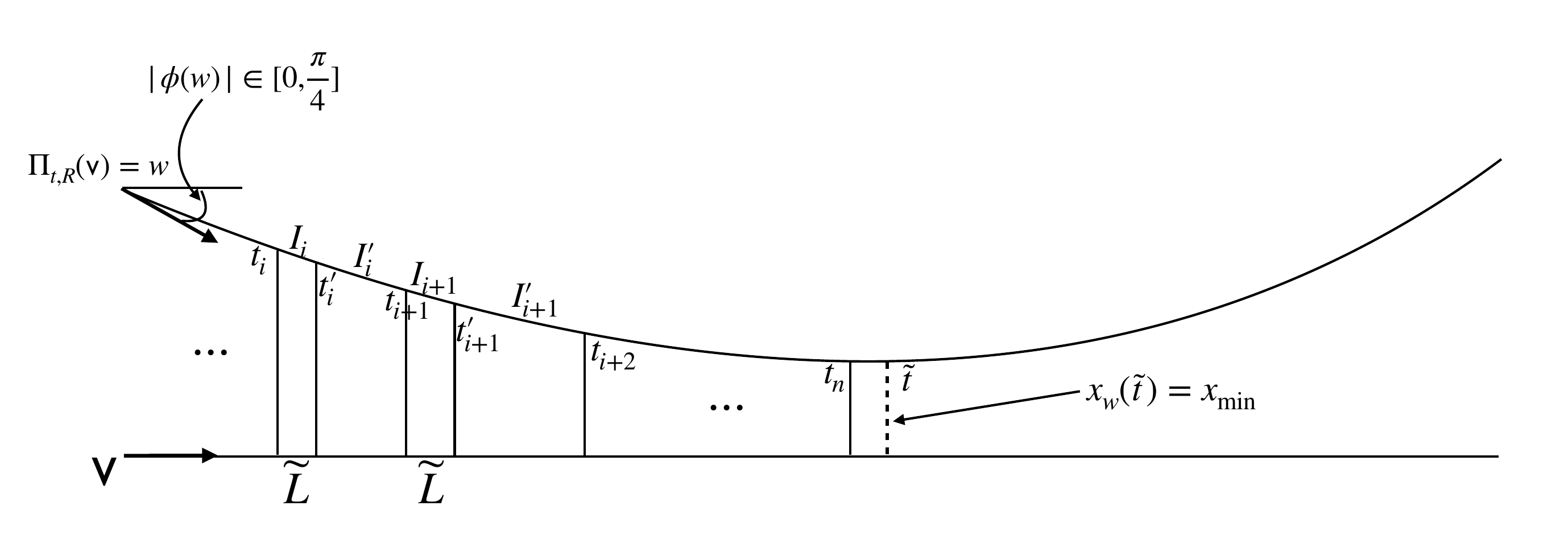} \caption{Division into $I_{i}$ and $I_{i}'$.}
\label{fig3} 
\end{figure}

Since the angle $\phi_{w}(\tau)$ satisfies $|\phi_{w}(\tau)|\in[0,\pi/4]$
for any $\tau\in[0,t]$ from Lemma \ref{lem: angle_lower_bdd_reg},
there exists $N\in\mathbb{N}$ so that 
\[
\min\limits _{i}|I_{i}|\geq\frac{1}{N}\max\limits _{i}|I'_{i}|
\]
where the maximum is taken over all possible $i$ and the minimum
is taken over $i$ in $\{1,\ldots,n\}$ if $\wt{t}\in I_{n}'$ (i.e.
$I_{n}'$ is nonempty) or else (i.e. $\wt{t}\in I_{n}$ and $I_{n}'$
is empty) in $\{1,\ldots,n-1\}$ in order to exclude $|I_{n}|$ which
could be arbitrarily small.

The following lemma shows $x_{w}''(\tau)$ admits a similar bound
as in Lemma \ref{lem: grow_rate_x''_HD} when $\tau$ belongs to $I_{i}$
for some $i$. 
\begin{lem}
\label{lem: grow_rate_x''_GW} There exists $C_{0}>0$ independent
of $w$ such that 
\[
0\leq x_{w}''\leq C_{0}x_{w}^{m+1}
\]
for all $\tau\in[0,\wt{t}\,]$. Moreover, there exists $c_{0}\in(0,C_{0})$
such that whenever $\tau\in I_{i}$ for some $i$, we also have the
lower bound 
\[
x_{w}''\geq c_{0}x_{w}^{m+1}.
\]
\end{lem}

\begin{proof} In what follows, we adopt the notation from Remark \ref{rem: surface_metric}. Since $G_{xx}=-KG\leq2C_{1}x_{w}^{m}$ with $G_{x}(s,0)=0$ for all
$s$, thus $G_{x}(s,x_{w})\leq2C_{1}(m+1)^{-1}x_{w}^{m+1}$. Let $C_{0}:=2C_{1}(m+1)^{-1}$.
By \eqref{eqn: geod_eqn}, 
\[
x''_{w}=\frac{G_{x}}{G}\cos^{2}\phi_{w}\leq\frac{G_{x}}{G}\leq C_{0}x_{w}^{m+1}.
\]
For $t\in I_{i}$, we have $G_{xx}=-KG\geq C_{2}x_{w}^{m}$, thus
$G_{x}(s,x_{w})\geq C_{2}(m+1)^{-1}x_{w}^{m+1}$. Let $c_{0}:=C_{2}(2m+2)^{-1}$.
By Lemma \ref{lem: angle_lower_bdd_reg} (1), we have 
\[
x''_{w}=\frac{G_{x}}{G}\cos^{2}\phi_{w}\geq\frac{G_{x}}{2G}\geq c_{0}x_{w}^{m+1}.
\]
This completes the proof. 
\end{proof}
For the following lemma we let ${\displaystyle t''_{i}:=\frac{t_{i}+t'_{i}}{2}}$.
For simplicity, in the remaining part of this section, we abbreviate
$x=x_{w}$ and $\phi=\phi_{w}$ where $w=\Pi_{t,R}(\v)$ for any $\v\in\Sing$
and $t,R>0$. 
\begin{lem}
\label{lem: dist_conv} For any $i$ with $|I_{i}|\geq\gamma_{1}$
(namely, those $I_{i}$ not containing 0 or $\wt{t}$), we have

\begin{enumerate}
\item ${\displaystyle \int_{t''_{i}}^{t'_{i}}-x'x^{m+1}ds\geq\frac{1}{2N+1}\int_{t''_{i}}^{t_{i+1}}-x'x^{m+1}ds.}$ 
\item ${\displaystyle \int_{\tau}^{t'_{i}}-x'x^{m+1}ds\geq\frac{1}{2N+1}\int_{\tau}^{t_{i+1}}-x'x^{m+1}ds}$
for any $\tau\in[t_{i},t''_{i}]$. 
\end{enumerate}
\end{lem}

\begin{proof}
For (1) consider any $s\in[t''_{i},t'_{i}]$ and $s'\in[t'_{i},t_{i+1}]$.
Since $x$ is convex and decreasing, we have $-x'(s)\geq-x'(s')$
and $x(s)\geq x(s')$. In particular, 
\[
-x'(s)x(s)^{m+1}\geq-x'(s')x(s')^{m+1}.
\]
Since ${\displaystyle \min_{i}|I_{i}|\geq\frac{1}{N}\max_{i}|I'_{i}|}$,
we can divide $I'_{i}$ into $2N$ subintervals of the same length.
The integral on each subinterval, whose length is at most $|I_{i}|/2$,
is no more than that on $[t''_{i},t'_{i}]$. Hence 
\begin{eqnarray*}
\int_{t''_{i}}^{t_{i+1}}-x'x^{m+1}ds & = & \int_{t''_{i}}^{t'_{i}}-x'x^{m+1}ds+\int_{t'_{i}}^{t_{i+1}}-x'x^{m+1}ds\\
 & \leq & \int_{t''_{i}}^{t'_{i}}-x'x^{m+1}ds+2N\int_{t''_{i}}^{t'_{i}}-x'x^{m+1}ds\\
 & = & (1+2N)\int_{t''_{i}}^{t'_{i}}-x'x^{m+1}ds.
\end{eqnarray*}

For (2), by (1), we have 
\begin{eqnarray*}
\int_{\tau}^{t'_{i}}-x'x^{m+1}ds & = & \int_{\tau}^{t''_{i}}-x'x^{m+1}ds+\int_{t''_{i}}^{t'_{i}}-x'x^{m+1}ds\\
 & \geq & \int_{\tau}^{t''_{i}}-x'x^{m+1}ds+\frac{1}{2N+1}\int_{t''_{i}}^{t_{i+1}}-x'x^{m+1}ds\\
 & \geq & \frac{1}{2N+1}\int_{\tau}^{t_{i+1}}-x'x^{m+1}ds.
\end{eqnarray*}
\end{proof}
\begin{lem}
\label{lem: ODE_half_int} There exists $C_{3}>0$ such that for any
$i$ and $\tau\in[t_{i},t''_{i}]$, 
\[
x'(\tau)\leq-2C_{3}\sqrt{x(\tau)^{m+2}-x(\,\wt{t}\,)^{m+2}}.
\]
\end{lem}

\begin{proof}
If $\wt{t}\in I'_{n}$, by Lemmas \ref{lem: grow_rate_x''_GW} and
\ref{lem: dist_conv}, we have 
\begin{eqnarray*}
x'(\tau)^{2} & = & \int_{\tau}^{\wt{t}}-2x'x''ds\geq2c_{0}\left(\int_{\tau}^{t'_{i}}-x'x^{m+1}ds+\sum_{k=i+1}^{n}\int_{t_{k}}^{t'_{k}}-x'x^{m+1}ds\right)\\
 & \geq & \frac{2c_{0}}{2N+1}\left(\int_{\tau}^{t_{i+1}}-x'x^{m+1}ds+\sum_{k=i+1}^{n-1}\int_{t_{k}}^{t_{k+1}}-x'x^{m+1}ds+\int_{t_{n}}^{\wt{t}}-x'x^{m+1}ds\right)\\
 & \geq & \frac{2c_{0}}{2N+1}\int_{\tau}^{\wt{t}}-x'x^{m+1}ds\geq\frac{2c_{0}}{(m+2)(2N+1)}(x(\tau)^{m+2}-x(\,\wt{t}\,)^{m+2}).
\end{eqnarray*}

Similarly, if $\wt{t}\in I_{n}$, we have 
\begin{eqnarray*}
x'(\tau)^{2} & = & \int_{\tau}^{\wt{t}}-2x'x''ds\\
 & \geq & 2c_{0}\left(\int_{\tau}^{t'_{i}}-x'x^{m+1}ds+\sum_{k=i+1}^{n-1}\int_{t_{k}}^{t'_{k}}-x'x^{m+1}ds+\int_{t_{n}}^{\wt{t}}-x'x^{m+1}ds\right)\\
 & \geq & \frac{2c_{0}}{2N+1}\left(\int_{\tau}^{t_{i+1}}-x'x^{m+1}ds+\sum_{k=i+1}^{n-1}\int_{t_{k}}^{t_{k+1}}-x'x^{m+1}ds+\int_{t_{n}}^{\wt{t}}-x'x^{m+1}ds\right)\\
 & \geq & \frac{2c_{0}}{2N+1}\int_{\tau}^{\wt{t}}-x'x^{m+1}ds\geq\frac{2c_{0}}{(m+2)(2N+1)}(x(\tau)^{m+2}-x(\,\wt{t}\,)^{m+2}).
\end{eqnarray*}
We finish the proof by taking $C_{3}:=2^{-1/2}c_{0}^{1/2}(m+2)^{-1/2}(2N+1)^{-1/2}$. 
\end{proof}
\begin{proof}[Proof of Proposition \ref{prop: estm_x_psi_GW}]
As did in Proposition \ref{prop: estm_x_psi}, we will use $Q$ to
denote a generic constant. The desired lower bound for $x(\tau)$
can be established just as done in Proposition \ref{prop: estm_x_psi}.
This is because the upper bound $0\leq x_{w}''\leq C_{0}x_{w}^{m+1}$
from Lemma \ref{lem: grow_rate_x''_GW} holds for all $\tau\in[0,\wt{t}\,]$,
and this is the only ingredient needed for the lower bound on $x(\tau)$
in Proposition \ref{prop: estm_x_psi}. In particular, we have $x(\tau)\geq(R^{-\frac{m}{2}}+\sqrt{C_{0}}\tau)^{\frac{2}{m}}$
for all $\tau\in[0,\wt{t}\,]$.

On the other hand, the desired upper bound for $x(\tau)$ is more
difficult to obtain. The reason for introducing $t_{i}''$ and establishing
Lemma \ref{lem: ODE_half_int} was to obtain the upper bound. We will
prove the case where $0\in I'_{0}$ and $\wt{t}\in I_{n}$. Other
cases are similar, and we will comment on them at the end of the proof.

First, define a sequence 
\[
S_{n}:=\Big(\sum_{k=1}^{n-1}|t''_{k}-t_{k}|\Big)+|\wt{t}-t_{n}|,~S_{0}=0,\text{ and }S_{i}=\sum_{k=1}^{i}|t''_{k}-t_{k}|.
\]

Then we have $0=S_{0}<S_{1}<\cdots<S_{n-1}<S_{n}$. Define a function
$f:[0,S_{n}]\to\R$ via 
\[
f(\tau):=x(\tau-S_{i}+t_{i+1})\text{ when }\tau\in(S_{i},S_{i+1}].
\]
Then $f$ is a piecewise smooth function with discontinuities at each
$S_{i}$. Moreover, Lemma \ref{lem: ODE_half_int} shows that $f$
is strictly decreasing function satisfying the assumption of Lemma
\ref{lem: estm_sol_ODE_discont} with $Q=C_{3}$. Thus by Lemma \ref{lem: estm_sol_ODE_discont},
there exists $Q_{0}>1$ such that 
\[
f(\tau)\leq Q_{0}(f(0)^{-\frac{m}{2}}+C_{3}m\tau)^{-\frac{2}{m}}\leq Q_{0}(R^{-\frac{m}{2}}+C_{3}m\tau)^{-\frac{2}{m}}
\]
for all $\tau\in[0,S_{n}]$. In particular, this inequality provides
an upper bound for $x(\tau)$ for $\tau\in[t_{i},t_{i}'']$ for $1\leq i\leq n$;
here $t_{n}''$ should be replaced by $\wt{t}$ when $i=n$.

For $\tau\in[t_{i}'',t_{i+1}]$, we have ${\displaystyle \frac{S_{i}}{\tau}\geq\frac{1}{2N+2}}$
from the choice of $N$. Thus for $\tau\in[t_{i}'',t_{i+1}]$, 
\begin{eqnarray*}
x(\tau) & \leq & x(t_{i}'')=f(S_{i})\leq Q_{0}(f(0)^{-\frac{m}{2}}+C_{3}mS_{i})^{-\frac{2}{m}}\\
 & \leq & Q_{0}\left(R^{-m/2}+\frac{C_{3}m}{2N+2}\tau\right)^{-\frac{2}{m}}.
\end{eqnarray*}

For the last remaining subset of the domain when $\tau\in[0,t_{1}]$,
which is due to the assumption that $0\in I'_{0}$, we have $t_{1}\leq\sqrt{2}|\gamma_{0}|$.
Therefore, 
\[
x(\tau)\leq x(0)=R\leq R(\sqrt{2}|\gamma_{0}|+1)^{\frac{2}{m}}(\tau+1)^{-\frac{2}{m}}.
\]

In sum, we can find $Q>0$ such that 
\[
x(\tau)\leq Q(\tau+1)^{-\frac{2}{m}}
\]
for all $\tau\in[0,\wt{t}\,]$.

For $\phi$, using \eqref{eqn: phi_eqn_HD} and Lemmas \ref{lem: angle_lower_bdd_reg}
(1) and \ref{lem: grow_rate_x''_GW}, there exists $Q>0$ such that
\[
\phi'\leq\sqrt{2}\phi'\cos\phi=\sqrt{2}x''\leq\sqrt{2}C_{0}x^{m+1}\leq Q(\tau+1)^{-\frac{2(m+1)}{m}}.
\]
Thus we get the required upper bound for $\phi$: 
\[
|\phi(\tau)|=\left|\int_{\tau}^{\wt{t}}\phi'ds\right|\leq\int_{\tau}^{\wt{t}}|\phi'|ds\leq Q\int_{\tau}^{\wt{t}}(s+1)^{-\frac{2(m+1)}{m}}ds\leq Q[(\tau+1)^{-\frac{m+2}{m}}-(\wt{t}+1)^{-\frac{m+2}{m}}].
\]
This completes the proof when $0\in I'_{0}$ and $\wt{t}\in I_{n}$.

Other remaining cases can be dealt with similarly. When $0\in I_{0}$ and
$\wt{t}\in I_{n}$, then exactly the same proof works; in fact, there
is no need to separately consider $\tau\in[0,t_{1}]$ like we did
above. In the case where $\wt{t}\in I_{n}'$, we can proceed just
as we did above by bounding $x(\tau)$ above by $x(t_{n}'')$ for
$\tau\in[t_{n}'',\wt{t}\,]$.

Now, we consider the lower bound of $|\phi|$. For any $\tau\in[0,\wt{t}-2\sqrt{2}\gamma_{0}]$,
the interval $[\tau,\wt{t}\,]$ contains at least one $[t_{i},t'_{i}]$.
Let $l$ (resp. $L$) be the minimal (resp. maximal) $i$ with $[t_{i},t'_{i}]\subset[\tau,\wt{t}\,]$.
We firstly compare the integrals of $x(s)^{m+1}$ on $[\tau,t_{l}]$
and $[t_{l},t_{l+1}]$. Since $|\phi|<\pi/4$, 
\begin{equation}
t_{l+1}-t_{l}\geq\gamma_{0}\geq\frac{t_{l}-\tau}{\sqrt{2}}.\label{eqn: comp_adj_int_length}
\end{equation}
Moreover, we have 
\[
\frac{t_{l+1}+1}{\tau+1}=1+\frac{t_{l+1}-\tau}{\tau+1}\leq1+2\sqrt{2}\gamma_{0}.
\]
Hence 
\[
x(t_{l+1})\geq Q(t_{l+1}+1)^{-2/m}\geq Q(\tau+1)^{-2/m}\geq Qx(\tau).
\]
Together with \eqref{eqn: comp_adj_int_length}, since $x$ is non-increasing,
we get 
\begin{equation}
\int_{t_{l}}^{t_{l+1}}x(s)^{m+1}ds\geq(t_{l+1}-t_{l})x(t_{l+1})^{m+1}\geq Q\frac{t_{l}-\tau}{\sqrt{2}}x(\tau)^{m+1}\geq Q\int_{\tau}^{t_{l}}x(s)^{m+1}ds.\label{eqn: comp_adj_int}
\end{equation}
Notice that $\phi'=x''/\cos\phi\geq x''$, and $|I_{i}|/|I_{j}|\in[1/\sqrt{2},\sqrt{2}]$
for any $i,j$. Thus, by \eqref{eqn: comp_adj_int} and Lemma \ref{lem: dist_conv}, 
\begin{eqnarray*}
|\phi(\tau)| & = & \int_{\tau}^{\wt{t}}\phi'(s)ds\geq\int_{\tau}^{\wt{t}}x''(s)ds\geq c_{0}\sum_{i=l}^{L}\int_{t_{i}}^{t'_{i}}x(s)^{m+1}\,ds\\
 & \geq & \frac{c_{0}}{2N+1}\sum_{i=l}^{L-1}\int_{t_{i}}^{t_{i+1}}x(s)^{m+1}ds+\frac{c_{0}}{2N+1+\sqrt{2}}\int_{t_{L}}^{\wt{t}}x(s)^{m+1}ds\\
 & \geq & Q\int_{\tau}^{\wt{t}}x(s)^{m+1}ds\geq Q\int_{\tau}^{\wt{t}}(s+1)^{-2(m+1)/m}ds\\
 & \geq & Q[(\tau+1)^{-\frac{m+2}{m}}-(\wt{t}+1)^{-\frac{m+2}{m}}].
\end{eqnarray*}
\end{proof}

\section{Geometric potentials\label{sec:geo_potentials}}

This section aims to prove Theorem \ref{Main-result-2}.
Let $M$ be a closed rank 1 nonpositively curved manifold, and ${\cal F=}(f_{t})_{t\in\R}$
denote the geodesic flow on $T^{1}M.$ Recall that the geometric potential
is defined via 
\[
\varphi^{u}(v):=-\lim_{t\to0}\frac{1}{t}\log\det(df_{t}|_{E^{u}(v)})=-\frac{d}{dt}\Big|_{t=0}\log\det(df_{t}|_{E^{u}(v)}).
\]
As indicated in \cite[Section 7.2]{burns2018unique}, it is convenient
to consider the following auxiliary function whose time evolution
is governed by a Riccati equation:

\[
\psi^{u}(v):=-\lim_{t\to0}\frac{1}{t}\log\det(J_{v,t}^{u})=-\frac{d}{dt}\Big|_{t=0}\log\det(J_{v,t}^{u}),
\]
where $J_{v,t}^{u}:w\in v^{\perp}\mapsto J(t)\in(f_{t}v)^{\perp}$
and $J(t)$ is a unstable Jacobi field along $\gamma_{v}$ such that
$J(0)=w$. We also have $\psi^{u}(v)=-\text{tr }U_{v}^{u}(0)$ where
$U_{v}^{u}(t)$ is the shape operator of the unstable horoshpere $H^{u}(f_{t}v)$.

Let $\pi:T^{1}M\to M$ be the canonical projection. Its derivative
$d\pi_{v}:T_{v}T^{1}M\to T_{\pi v}M$ sends $E^{u}(v)$ onto $v^{\perp}.$
We have $df_{t}=d\pi_{f_{t}v}^{-1}\circ J_{v,t}^{u}\circ d\pi_{v}$,
and thus 
\[
\det(df_{t}|_{E^{u}(v)})=\det(d\pi_{f_{t}v})^{-1}\det(J_{v,t}^{u})\det(d\pi_{v}).
\]
Thus 
\begin{equation}
\varphi^{u}(v)-\psi^{u}(v)=\frac{d}{dt}\Big|_{t=0}\log\det(d\pi_{f_{t}v}).\label{eqn: diff_phi_psi_1st}
\end{equation}
For any $t$, since $U_{v}^{u}(t)$ is symmetric, we can take an orthonormal
basis $\{e_{i}(t)\}_{i=1}^{n-1}$ of $v^{\perp}$ so that $U_{v}^{u}(t)e_{i}(t)=\lambda_{i}(t)e_{i}(t)$
with $\lambda_{i}(t)\geq0$. Since $||df_{t}(\xi)||_{S}=\|J_{\xi}(t)\|^{2}+\|J'_{\xi}(t)\|^{2}$
for $\xi\in E_{v}^{u}$, for any $t$, we have an orthonormal basis
$\{\xi_{i}(t)\}_{i=1}^{n-1}$ of $E_{f_{t}v}^{u}$, where $\xi_{i}(t)$
is determined by 
\[
J_{\xi_{i}}(t)=\frac{e_{i}(t)}{\sqrt{1+\lambda_{i}(t)^{2}}},\quad J'_{\xi_{i}}(t)=\frac{\lambda_{i}(t)e_{i}(t)}{\sqrt{1+\lambda_{i}(t)^{2}}}.
\]
Thus $d\pi_{f_{t}v}(\xi_{i}(t))=e_{i}(t)/\sqrt{1+\lambda_{i}(t)^{2}}$
and the matrix of $d\pi_{f_{t}v}$ with respect to these two orthonormal
bases is $\text{diag}((1+\lambda_{1}(t)^{2})^{-1/2},\cdots,(1+\lambda_{n-1}(t)^{2})^{-1/2})$.
Hence 
\begin{equation}
\log\det(d\pi_{f_{t}v})=\log\prod_{i=1}^{n-1}(1+\lambda_{i}(t)^{2})^{-1/2}=-\frac{1}{2}\log\det(I_{n-1}+U_{v}^{u}(t)^{2}).\label{eqn: log_det_dpi}
\end{equation}
Now we use the following Jacobi formula for $A:\mathbb{R}\to M_{n-1}$:
\[
\frac{d}{dt}\det A(t)=\det A(t)\text{ tr}\left(A(t)^{-1}\frac{dA(t)}{dt}\right).
\]

For simplicity, denote by $U=U_{v}^{u}$ and $I=I_{n-1}$. By \eqref{eqn: diff_phi_psi_1st}
and \eqref{eqn: log_det_dpi}, we have 
\begin{eqnarray*}
\varphi^{u}(v)-\psi^{u}(v) & = & -\frac{1}{2}\text{ tr}\left[(I+U(0)^{2})^{-1}(U'(0)U(0)+U(0)U'(0))\right]\\
 & = & -\frac{1}{2}\text{ tr}\left[(U(0)(I+U(0)^{2})^{-1}+(I+U(0)^{2})^{-1}U(0))U'(0)\right]\\
 & = & -\text{ tr}\left[U(0)(I+U(0)^{2})^{-1}U'(0)\right]
\end{eqnarray*}
Since $U'+U^{2}+\mathcal{K}=0$, $U$ is positive semidefinite, and
$|\text{tr}(AB)|\leq|\text{tr}(A)||\text{tr}(B)|$ if $A,B$ are positive
semidefinite, we have 

\begin{eqnarray*}
|\varphi^{u}(v)-\psi^{u}(v)| & = & \left|\left( \text{ tr}\left[U^{3}(0)(I+U(0)^{2})^{-1}\right]-\text{ tr}\left[U(0)(I+U(0)^{2})^{-1}(-\mathcal{K})\right] \right)\right| \\
 & \leq & (\text{tr }U(0))^{3}+\text{ tr}[U(0)(I+U(0)^{2})^{-1}]\text{ tr}(-\mathcal{K})\\
 & = & -\psi^{u}(v)^{3}-\psi^{u}(v)(-\text{Ric}(v))\leq-\psi^{u}(v)(\psi^{u}(v)^{2}-\text{Ric}(v)),
\end{eqnarray*}

When $v$ is sufficiently close to Sing, $-\psi^{u}(v)$ and $-\text{Ric}(v)$
are small nonnegative numbers, thus we have $\varphi^{u}(v)\approx\psi^{u}(v)$
near $\Sing$. We summarize the above discussion below: 
\begin{prop}
\label{prop:comparability} Suppose $M$ is a closed rank 1 nonpositively
curved manifold. Then we have 
\[
|\varphi^{u}(v)-\psi^{u}(v)|\leq-\psi^{u}(v)(\psi^{u}(v)^{2}-\text{Ric}(v)).
\]
In particular, we have $\vp^{u}(v)\approx\psi^{u}(v)$ near $\Sing$. 
\end{prop}

\subsection{The proof of Theorem \ref{Main-result-2}}

The strategy of the proof is to study the auxiliary function $\psi^{u}$
through the associated Riccati equation. We establish a version of
Theorem \ref{Main-result-2} for $\psi^{u}$. Then Theorem \ref{Main-result-2}
follows Proposition \ref{prop:comparability}.

We remark that the additional Ricci curvature constraint is essential
in our argument. In the higher dimension scenario, only having radial curvature controlled is insufficient. Nevertheless, for some special
Riemannian metrics, namely, warped products, the radial curvature
$K_{\perp}(v)$ and Ricci curvature ${\rm Ric}(v)$ are comparable.
Since this observation is not in the mainstream of the current paper,
we leave the proof in Appendix \ref{sec:Appendix}.

Let $M$ be a Type 1 manifold with order $m-1$ Ricci curvature bounds,
that is, there exists $k_{0},K_{0}>0$ such that 
\begin{equation}
-K_{0}|x_{v}|^{m}\leq\text{Ric}(v)\leq-k_{0}|x_{v}|^{m}\label{eqn: ric_curv_bdd}
\end{equation}
for all $v$ with $|x_{v}|\leq\ep$. 

\begin{prop}
\label{prop: hold_cont_geom_pot} Assume $M$ satisfies the assumption of Theorem \ref{Main-result-2},  then 
$$-\psi^{u}(v)\approx|x_{v}|^{m/2}+|\phi_{v}|^{m/(m+2)},$$
for any $v$ near Sing.
\end{prop}
In particular,  we have the same scale estimation for $-\varphi^{u}$,  and Theorem  \ref{Main-result-2} follows.  
To prove Proposition \ref{prop: hold_cont_geom_pot}, we need the
the following lemma. 
\begin{lem}
\label{lem: est_trace_ric_curv} Assume there exist $K_{1}>k_{1}>0$
so that 
\begin{enumerate}[font=\normalfont]
\item $-K_{1}^{2}T^{-2}\leq\text{Ric}(\gamma'_{v}(t))\leq-k_{1}^{2}T^{-2}$
for all $t\in[-T,0]$, then there exists $K_{2}>k_{2}>0$ depending
on $k_{1},K_{1}$ so that 
\[
k_{2}T^{-1}\leq-\psi^{u}(v)\leq K_{2}T^{-1}.
\]
\item $-K_{1}^{2}T^{-2}\leq\text{Ric}(\gamma'_{v}(t))\leq0$ for all $t\in[0,k_{1}T]$,
and $k_{3}T^{-1}\leq-\psi^{u}(v)\leq K_{3}T^{-1}$, for some $k_3<K_3$, then there exist
$K_{4}>k_{4}>0$ depending on $k_{1},K_{1},k_{3},K_{3}$ so that 
\[
k_{4}T^{-1}\leq-\psi^{u}(\gamma'_{v}(t))\leq K_{4}T^{-1},
\]
for all $t\in[0,k_{1}T].$ 
\end{enumerate}
\end{lem}

\begin{proof}
Denote by $u(t)=\frac{1}{n-1}\text{tr}(U_{v}^{u}(t))$. Since $U_{v}^{u}$
is diagonalizable and all eigenvalues $\lambda_{i}(t)$ are nonnegative,
by Cauchy-Schwartz we have 
\[
\text{tr}((U_{v}^{u})^{2})\leq(n-1)^{2}u^{2}\leq(n-1)\text{tr}((U_{v}^{u})^{2}).
\]
Thus, by the Riccati equation, 
\[
u'=\frac{\text{tr}((U_{v}^{u})')}{n-1}=-\frac{\text{tr}((U_{v}^{u})^{2})}{n-1}-\frac{\text{Ric}(\dot{\gamma})}{n-1}\leq-u^{2}-\frac{\text{Ric}(\dot{\gamma})}{n-1}.
\]
On the other hand, denote by $w(t)=(n-1)u(t)=\text{tr}(U_{v}^{u}(t)).$
We have 
\[
w'=-\text{tr}((U_{v}^{u})^{2})-\text{Ric}(\dot{\gamma})\geq-w^{2}-\text{Ric}(\dot{\gamma}).
\]
\begin{enumerate}
\item Compare $u$ with the solution of 
\[
\bar{u}'+\bar{u}^{2}-K_{1}^{2}T^{-2}=0,\,\,\bar{u}(-T)=+\infty\quad\Rightarrow\quad\bar{u}(0)=K_{1}T^{-1}\coth K_{1}.
\]
By the main theorem in \cite{eschenburg1990comparison}, we have 
\[
-\psi^{u}(v)=(n-1)u(0)\leq(n-1)\bar{u}(0)=:K_{2}T^{-1}.
\]
Compare $w$ with the solution of 
\[
\bar{w}'+\bar{w}^{2}-k_{1}^{2}T^{-2}=0,\,\,\bar{w}(-T)=0\quad\Rightarrow\quad\bar{w}(0)=k_{1}T^{-1}\tanh k_{1}.
\]
We have 
\[
-\psi^{u}(v)=w(0)\geq\bar{w}(0)=:k_{2}T^{-1}.
\]
\item Compare $u$ with the solution of 
\[
\bar{u}'+\bar{u}^{2}-K_{1}^{2}T^{-2}=0,\,\,\bar{u}(0)=K_{3}T^{-1}.
\]
We have 
\[
\bar{u}(t)=K_{1}T^{-1}\coth(K_{1}T^{-1}t+\coth^{-1}(K_{3}/K_{1})).
\]
Since $\coth$ is decreasing, so does $\bar{u}$. For $t\in[0,k_{1}T]$,
we get 
\[
-\psi^{u}(\gamma'_{v}(t))=(n-1)u(t)\leq(n-1)\bar{u}(t)\leq(n-1)\bar{u}(0)=:K_{4}T^{-1}.
\]
Compare $w$ with the solution of 
\[
\bar{w}'+\bar{w}^{2}=0,\,\,\bar{w}(0)=k_{3}T^{-1}\quad\Rightarrow\quad\bar{w}(t)=(t+k_{3}^{-1}T)^{-1}.
\]
For $t\in[0,k_{1}T]$, 
\[
-\psi^{u}(\gamma'_{v}(t))=w(t)\geq\bar{w}(t)\geq(k_{1}T+k_{3}^{-1}T)^{-1}=:k_{4}T^{-1}.
\]
\end{enumerate}
\end{proof}
\begin{figure}
\subfloat[Case 1]{\includegraphics[scale=0.35]{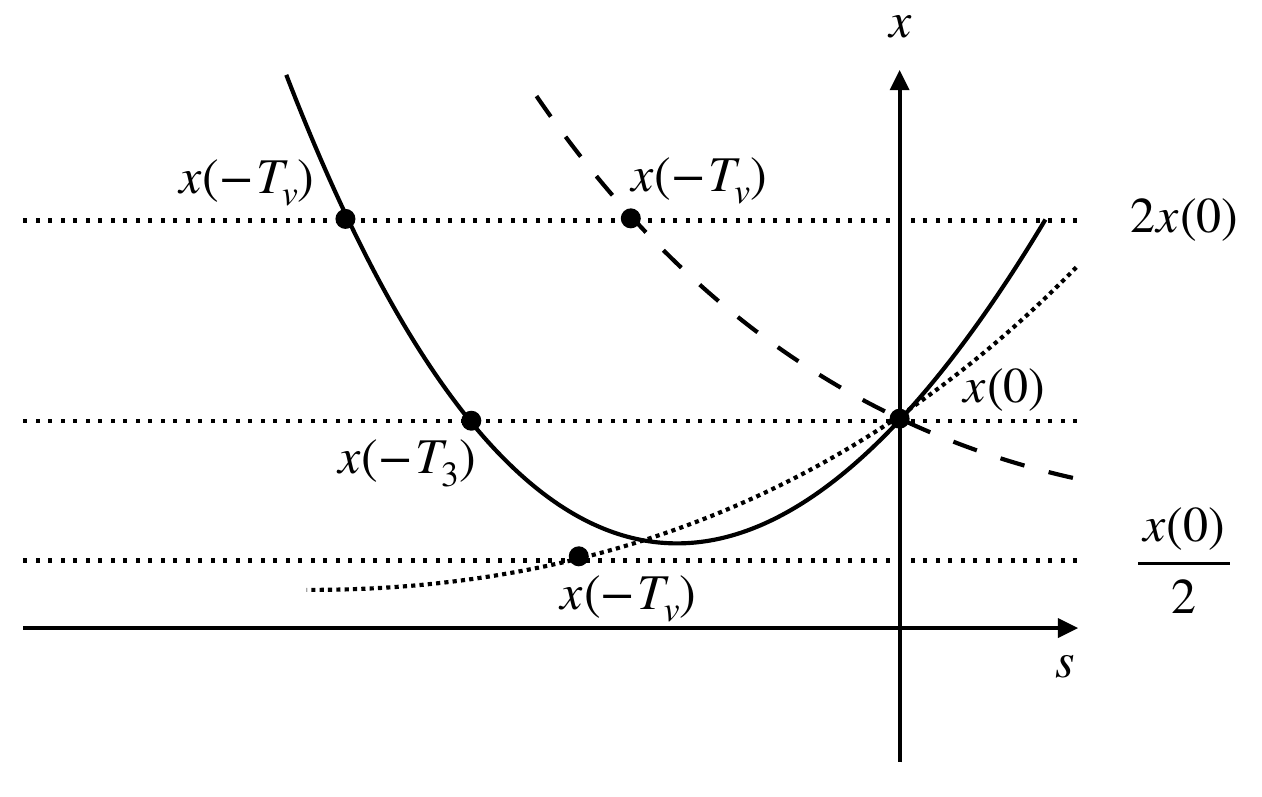}\label{fig:case1}

}\subfloat[Case 2]{\includegraphics[scale=0.35]{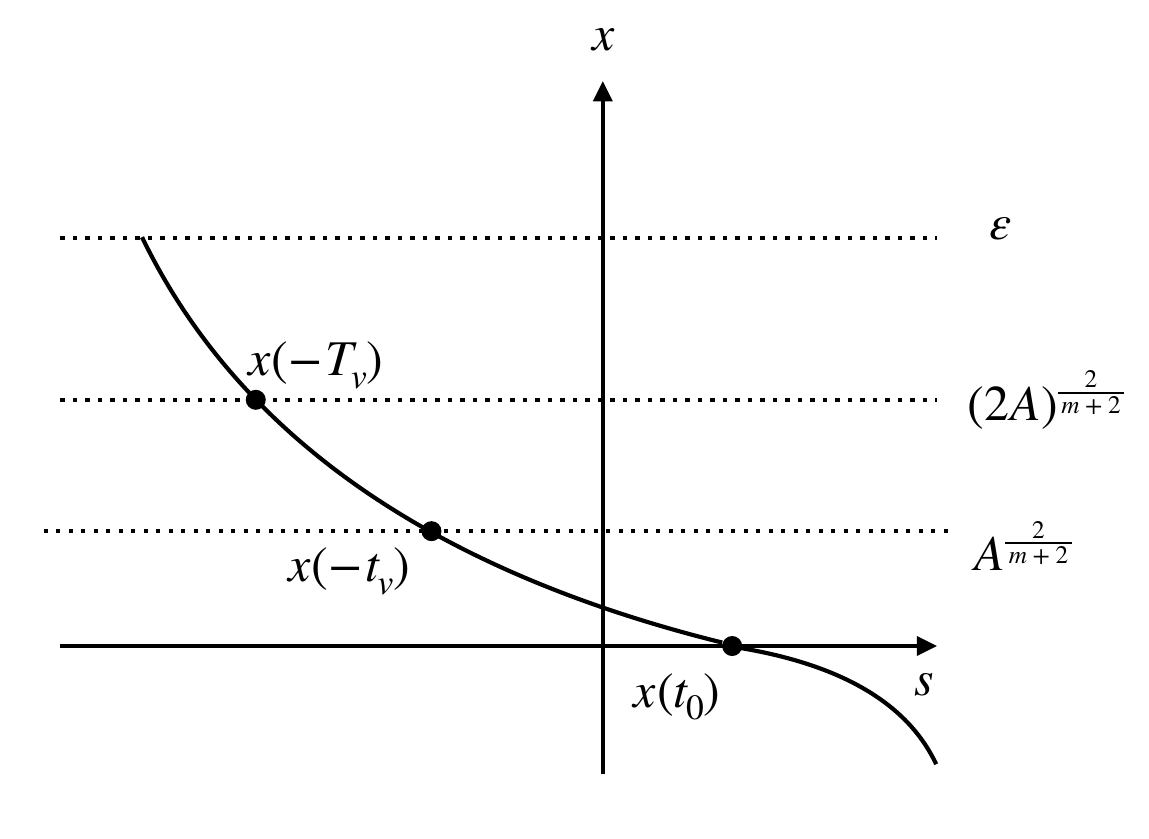}\label{fig:case2}

}

\caption{Proof of Proposition \ref{prop: hold_cont_geom_pot} }
\end{figure}

\begin{proof}[Proof of Proposition \ref{prop: hold_cont_geom_pot}]
We follow the main steps in the proof of \cite[Lemma 3.3]{gerber1999holder}.
We use $x,\phi$ instead of $x_{v},\phi_{v}$ for simplicity.\\
 \linebreak{}
\textbf{Case 1: $v$ is bouncing or asymptotic.} See Figure \ref{fig:case1}.
Since the asymptotic case is the bouncing case with $\tilde{t}=\infty$,
we only have to consider the bouncing $v$. We may assume $0<x(0)\leq\ep/2$.
Denote by 
\[
T_{v}:=\max\{T>0:x(t)\in[x(0)/2,2x(0)],\forall-T\leq t\leq0\}.
\]
For any $t\in[-T_{v},0]$, we have $-K_{5}x(0)^{m}\leq\text{Ric}(\dot{\gamma}(t))\leq-k_{5}x(0)^{m}$.
Moreover we have 
\begin{lem}
\label{lem: estim_Tv} $T_{v}\geq k_{6}x(0)^{-m/2}$ for some $k_{6}>0$
independent of $v$. 
\end{lem}

\begin{proof}
Case 1: $x(-T_{v})=2x(0)$, and $x(t)$ is decreasing on $[-T_{v},0]$.
By \eqref{eqn: upp_bdd_f_power} with $\alpha=m+2,\beta=1/2$, we
have 
\[
x(0)^{-m/2}\leq(2x(0))^{-m/2}+K_{6}T_{v}.
\]
Thus $T_{v}\geq k_{6}x(0)^{-m/2}$.

Case 2: $x(-T_{v})=x(0)/2$, similar to Case 1.

Case 3: $x(-T_{v})=2x(0)$, and $x(t)$ first decreases, then increases
on $[-T_{v},0]$. Assume $T_{3}\in(0,T_{v})$ satisfies $x(-T_{3})=x(0)$.
By Case 1 we know that $T_{v}>T_{v}-T_{3}\geq k_{6}x(0)^{-m/2}$. 
\end{proof}
\noindent Taking $T=k_{6}x(0)^{-m/2}$ in Lemma \ref{lem: est_trace_ric_curv}(1),
we get 
\[
k_{7}x(0)^{m/2}\leq-\psi^{u}(v)\leq K_{7}x(0)^{m/2}.
\]
By \eqref{eqn: bdd_x'_crossing-1}, we know that 
\[
|\phi(0)|^{m/(m+2)}\leq|2\sin\phi(0)|^{m/(m+2)}=(-2x'(0))^{m/(m+2)}\leq Qx(0)^{m/2}.
\]
Thus, we finish the proof of Proposition \ref{prop: hold_cont_geom_pot}
in this case. \\
 \linebreak{}
\textbf{Case 2: $v$ is crossing.} See Figure \ref{fig:case2}. Recall
that $x(t_{0})=0$, and $a(t)=-\sin\phi(t)$. Denote by $A:=|a(t_{0})|$.
Since $v$ is close to $\Sing$, we may assume that $2A<\ep^{(m+2)/2}$
and $0\leq x(0)<A^{2/(m+2)}$. Since $x'=\sin\phi$, by \eqref{eqn: bdd_x'_crossing-1}
we have 
\[
x(0)^{m/2}\leq A^{m/(m+2)}\leq|x'(0)|^{m/(m+2)}\leq|\phi(0)|^{m/(m+2)}.
\]
Since $x'=a$ and $x''\geq 0$ before crossing, we have $x'(0)\leq x'(t_0)=-A$ if $t_0>0$, and $x'(0)\geq x'(t_0)=A$ if $t_0<0$. Thus
\[
A\leq|x'(0)|=|\sin\phi(0)|\leq\sqrt{A^{2}+K_{8}x(0)^{m+2}}\leq A\sqrt{1+K_{8}}.
\]

\noindent Thus, it suffices to prove 
\begin{equation}
k_{9}A^{m/(m+2)}\leq-\psi^{u}(v)\leq K_{9}A^{m/(m+2)}.\label{eqn: psi_bdd_0}
\end{equation}
 We prove \eqref{eqn: psi_bdd_0} in the following two cases: 
\begin{itemize}
\item  Case 2a: $\phi(0)<0$.   In this case, $t_{0}>0$
and $x'<0$ for $t\in[0,t_{0}]$. Let $T_{v},t_{v}>0$ be the minimal
solutions of 
\[
x(-T_{v})=(2A)^{2/(m+2)},\quad x(-t_{v})=A^{2/(m+2)}.
\]

By the choice of $A$, we have $x(-t_{v})<x(-T_{v})<\ep$. Since $x'=\sin\phi$ and $x''\geq0$, thus $a(t)=|x'(t)|=|\sin\phi(t)|\geq|\sin\phi(t_{0})|=A$
for all $t\in[-t_{v},t_{0}]$. 
On the other hand, by \eqref{eqn: bdd_x'_crossing-1} we have
$$a^2\leq A^2+\frac{C_1}{(m+1)^2}x^{m+2}\leq A^2+\frac{C_1}{(m+1)^2}(2A)^2.$$
for all $t\in[-t_{v},t_{0}]$. In sum, there exists $K_{10}>0$ such that
\begin{equation}
A\leq a(t)\leq K_{10}A,\quad \text{for }t\in[-t_{v},t_{0}].\label{eqn: bdd_a}
\end{equation}
In \eqref{eqn: upp_bdd_f_power}, taking $a=-T_v$, $\tau=-t_v$, $f(t)=a(t)$, $\alpha=2$, and $\beta=\frac{m+1}{m+2}$ , we have 
\[
a(-t_v)^{-\frac{m}{m+2}}\leq a(-T_v)^{-\frac{m}{m+2}}+\frac{Q_2m}{m+2}(T_{v}-t_{v}).
\]
Together with \eqref{eqn: bdd_a}, we get $T_{v}-t_{v}>k_{10}A^{-m/(m+2)}$ for some $k_{10}>0$. Taking $T=k_{10}A^{-m/(m+2)}$
in Lemma \ref{lem: est_trace_ric_curv}(1), we have 
\begin{equation}
k_{11}A^{m/(m+2)}\leq-\psi^{u}(\gamma'_{v}(-t_{v}))\leq K_{11}A^{m/(m+2)}.\label{eqn: psi_bdd_tv}
\end{equation}
Use \eqref{eqn: bdd_a} again, we obtain
\begin{equation}
(t_{v}+t_{0})A\leq\int_{-t_{v}}^{t_{0}}-x'(s)ds=-x(t_{0})+x(-t_{v})=A^{2/(m+2)}.\label{eqn: bdd_t0_tv}
\end{equation}
Thus $t_{v}\leq A^{-m/(m+2)}$. By \eqref{eqn: psi_bdd_tv}, taking
$T=A^{-m/(m+2)}$ in Lemma \ref{lem: est_trace_ric_curv}(2), we get
\begin{equation}
k_{12}A^{m/(m+2)}\leq-\psi^{u}(v)\leq K_{12}A^{m/(m+2)}.\label{eqn: psi_bdd_0_pos}
\end{equation}
\item Case 2b: $\phi(0)>0$. In this case, we have $t_{0}<0$ since $\gamma_v$ crosses $T_0$ in the past.  Since $\gamma'_{v}(t_{0})$ crosses $T_0$ and we do not have flat strips, what happens for $t<t_0$ is covered in Case 2a (the absense of flat strip is necessary, see Remark \ref{rem: no_flat_strip_cond} for more details). By the result of Case 2a we
know that 
\begin{equation}
k_{12}A^{m/(m+2)}\leq-\psi^{u}(\gamma'_{v}(t_{0}))\leq K_{12}A^{m/(m+2)},\label{eqn: psi_bdd_0_neg}
\end{equation}
and $|t_{0}|\leq A^{-m/(m+2)}$ by \eqref{eqn: bdd_t0_tv}. 
Taking $T=A^{-m/(m+2)}$ again in Lemma \ref{lem: est_trace_ric_curv}(2),
we have 
\[
k_{13}A^{m/(m+2)}\leq-\psi^{u}(v)\leq K_{13}A^{m/(m+2)}.
\]
\end{itemize}

\end{proof}

\begin{rem} \label{rem: no_flat_strip_cond}

The absence of flat strips is crucial in Proposition \ref{prop: hold_cont_geom_pot} and Theorem \ref{Main-result-2}; otherwise, the H\"{o}lder continuity does not hold.  Here is a counterexample: consider a surface of revolution generated by 
$$f(x)=
 \begin{cases}
    |x+0.5|^{m+2}+1,   &  x<-0.5, \\
    1,
        &  -0.5\leq x\leq 0.5,\\
         (x-0.5)^{m+2}+1,   &  x>0.5.
 \end{cases}   
 $$
 The flat strip is the part with $-0.5\leq x\leq 0.5$,  and the metric satisfies both curvature conditions by Lemma \ref{lem: warped_product_curv}.  Let $v$ be a unit vector with $x=0.5$ and angle $\phi>0$,  meaning that $v$ is a vector exiting the flat strip.   At time $t=-\csc\phi$,  the geodesic $\gamma_v$ enters the strip with vector $v'=\gamma'_v(-\csc\phi)$.  By symmetry,  the angle of $v'$ is $-\phi$.  Assume that the  H\"{o}lder continuity in Proposition \ref{prop: hold_cont_geom_pot} holds for both $v$ and $v'$,  namely,  $\psi^u(v),  \psi^u(v')\approx -\phi^{m/(m+2)}$. 
As $\gamma_v(t)$ is in the flat strip for $t\in [-\csc\phi,  0]$,  $\psi^u(f_t v)$ satisfies the Riccati equation $u'+u^2=0$,  and the solution is 
$$\psi^u(f_t v)=(t+\psi^u(v)^{-1})^{-1}.$$
Plug in $\psi^u(v)\approx - \phi^{m/(m+2)}$ and $t=-\csc\phi$,  we have
$$\psi^u(v')=(-\csc\phi+\psi^u(v)^{-1})^{-1}\approx -\phi.$$
Contradictory to $\psi^u(v')\approx -\phi^{m/(m+2)}$. 

\end{rem}

The Hölder continuity of $\vp^{u}$ is an important,
yet still open, question in nonpositively curved geometry. Only some
partial results are known for surfaces under certain conditions, including
\cite[Lemma 3.3]{gerber1999holder} where Gerber and Wilkinson show
the Hölder continuity of $\vp^{u}$ for Type 2 surfaces. Since Ricci
curvature and Gaussian curvature are the same thing for surfaces,
using Theorem \ref{Main-result-2} we obtain a partial generalization
of \cite[Lemma 3.3]{gerber1999holder}: 
\begin{cor}
Under the same assumptions as \textup{Theorem \ref{Main-result-2}},
$\psi^{u}$ and $\vp^{u}$ are Hölder continuous in a small neighborhood
of $\Sing$. 
\end{cor}


\section{Sufficient criteria for the pressure gap\label{sec: pressure gap}}

Let $M$ be a closed Riemannian manifold and $\{f_{t}\}_{t\in\R}$
the geodesic flow on $T^{1}M$. In this section, we will describe
an abstract result to establish the pressure gap for a given potential
$\vp\colon T^{1}M\to\R$. But first, we need to introduce the notion
of specification in the following subsection.

\subsection{Specification}

While there are various definitions for it in the literature, roughly
speaking specification is a property that allows one to find an orbit
segment that shadows any given finite number of orbit segments at
a desired scale with controlled transition time. It was introduced
by Bowen \cite{bowen1974some} as one of the conditions to establish
the uniqueness of the equilibrium states for potentials over uniformly
hyperbolic maps. The specification still plays a vital role in many generalizations
of this result \cite{burns2018unique,chen2020unique,chen2021properties}.
The following version of the specification is from \cite[Theorem 4.1]{burns2018unique}. 
\begin{defn}[Specification]
We say a set of orbit segments $\CC$ satisfies the \textit{specification
at scale $\rho>0$} if there exists $\T>0$ such that given finite
orbit segments $(v_{1},t_{1}),\ldots,(v_{k},t_{k})\in\CC$ and $T_{1},\ldots,T_{k}\in\R$
with $T_{j+1}-T_{j}\geq t_{j}+\T$ for all $1\leq j\leq k-1$, there
is $w\in T^{1}M$ such that $f_{T_{j}}w\in B_{t_{j}}(v_{j},\rho)$
for all $1\leq j\leq k-1$. 
\end{defn}

This is a stronger version of the specification that appears in \cite{burns2018unique}
providing flexibility in the transition time. However, in practice,
we will always take $T_{j}$'s such that $T_{j+1}-T_{j}=t_{j}+\T$;
that is, the transition time is exactly equal to $\T$.

\subsection{Abstract result on the pressure gap}

\label{subsec: abstract result}

We now list the conditions together which establish the pressure gap.
Let $M$ be a closed rank-1 manifold containing a codimension-1 flat subtorus, with $\Sing$ induced by the subtorus. As mentioned in the introduction, one can easily extend results
in this section to multiple subtori scenarios.
However, for brevity, we stick to this simpler assumption.

By setting 
\[
\CC(\eta)=\{(v,t)\in T^{1}M\times\R^{+}\colon v,f_{t}v\in\Reg(\eta)\}
\]
to be the set of orbit segments with endpoints in $\Reg(\eta)$, we
require that the geodesic flow $\{f_{t}\}$ and the potential $\vp\colon T^{1}M\to\R$
satisfy the following properties: 
\begin{enumerate}
\item Singular set zero entropy property: $h_{\mathrm{top}}(\mathcal{F}|_{\Sing})=0$. 
\item Specification property: For any $\eta>0$ and $\rho>0$, the orbit
segments $\CC(\eta)$ satisfies the specification at scale $\rho$. 
\item Shadowing property: There exists $R>0$ such that for every $t>0$
there exists a map 
\[
\Pi_{t,R}\colon\Sing\to T^{1}M
\]
with the following properties: denoting by $w_{\v}:=\Pi_{t,R}(\v)$
the shadowing vector of an arbitrary $\v\in\Sing$, 
\begin{enumerate}
\item The conclusion of Lemma \ref{lem: angle_lower_bdd_reg} holds for $\Pi_{t,R}$. 
\item For any $\ep>0$, there exists $L:=L(\ep,R)$ such that for any $t>2L$,
the vector $w_{\v}$ satisfies 
\[
d(f_{\tau}w_{\v},\Sing)<\ep
\]
for all $\tau\in[L,t-L]$. 
\end{enumerate}
\item Special Bowen property: For any $\d>0$, there exists $C=C(\d,R)>0$
independent of $t$ such that 
\[
\int_{0}^{t}\vp(f_{\tau}u)d\tau-\int_{0}^{t}\vp(f_{\tau}\v)d\tau>-C
\]
for any $u\in B_{t}(w_{\v},\d)$. 
\end{enumerate}
\begin{prop}
\label{prop: abstract pressure gap} Suppose the geodesic flow $\{f_{t}\}_{t\in\R}$
and the potential $\vp$ satisfy the above listed conditions \textup{(1),
(2), (3)} and \textup{(4)}. Then, $\vp$ has a pressure gap. 
\end{prop}

Assuming we have Proposition \ref{prop: abstract pressure gap}, we finish the proof of Theorem \ref{Main-result}:

\begin{proof}[Proof of Theorem \ref{Main-result}]
Each condition listed above can be verified as follows. By design, we know $h_{\mathrm{top}}(\mathcal{F}|_{\Sing})=0$.
The specification property (2) is already established in \cite{burns2018unique} for the geodesic flow $\{f_{t}\}$ over rank 1 nonpositively curved manifold. With $\Pi_{t,R}$ defined as in \eqref{eq: Pi} via the Fermi coordinates, (3a) is immediate as we have already proved Lemma \ref{lem: angle_lower_bdd_reg}. For (3b), we can take $L$ to be $(Q/\ep)^{m/2}$ where $Q$ is the constant from Proposition \ref{prop: estm_x_psi} and \ref{prop: estm_x_psi_GW}.  Lastly,  (4) follows from Proposition \ref{prop: bowen_ball_int_esti}. Hence, $\vp$ has the pressure gap by the above proposition. 
\end{proof}

\textcolor{black}{We note that the proof of Proposition \ref{prop: abstract pressure gap}
draws inspiration from \cite[Theorem B]{burns2018unique}. However,
leveraging the singular set’s zero entropy property, Peres’ Lemma
allows us to circumvent several technicalities and arrive at a more
straightforward proof than that presented in \cite{burns2018unique}.
See Remark \ref{rem: difference} for more details. }
\begin{thm}
\cite[Lemma 2]{Peres88} \label{thm:Peres_Lemma} Let $\mathcal{F}=\{f_{t}\}$
be a continuous flow on a compact space $X$, and $\mu$ be an $\mathcal{F}-$invariant
probability measure. Then for every continuous potential $\vp:X\to\mathbb{R}$
there exists some $v\in X$ 
\[
\frac{1}{T}\int_{0}^{T}\vp(f_{\tau}v)d\tau\geq \int_{X}\vp d\mu
\]
for all $T>0$. 
\end{thm}

The authors believe this Peres' result is known among the experts;
however, we cannot find proof of the flow version
in the literature. For the completeness, we give a proof in the appendix;
see Theorem \ref{thm:proof-peres}.
\begin{proof}[Proof of Proposition \ref{prop: abstract pressure gap}]
This proof has three steps. The first step is using the singular
set zero entropy property (1) and Theorem \ref{thm:Peres_Lemma} to
bound $P(\vp,\Sing)$ by the integration of some special $\v$ in
$\Sing$ along the flow. The second step is using the specification
property (2) and shadowing property (3) to create a $(t,\delta)-$separated
set that bounds $P(\vp,\Sing)$. The last step is using the special
Bowen property (4) to estimate the pressure.

Notice that without loss of generality, we may assume $P(\vp,\Sing)=0$;
otherwise, we consider $\varphi-P(\vp,\Sing)$. Hence, it is sufficient
to show that for large $t$ there exists a $(t,\delta)-$separated set
$F_{t}$ on $T^{1}M$ such that 
\begin{equation}
\lim_{t\to\infty}\frac{1}{t}\sum_{w\in F_{t}}e^{\int_{0}^{t}\vp(f_{\tau}w)d\tau}>0.\label{eq: PSing_approx}
\end{equation}

The first step is to find a special singular vector $\v$ for shadowing.
Notice that \( h_{\mathrm{top}}(\mathcal{F}|_\Sing) = 0 \) implies that the function \( \mu \mapsto h_{\mu}(\mathcal{F}|_\Sing) \) is constantly zero for any invariant measure \( \mu \) and, in particular, is upper semicontinuous. Consequently, any weak-* limit of a sequence of invariant measures approximating the pressure is an equilibrium state. That is 
\[
0=P(\vp,\Sing)=h_{\mu}(\mathcal{F})+\int_{\Sing}\vp d\mu.
\]
By the singular set zero entropy property (1), we know $h_{\mu}(\mathcal{F})=0$;
and thus, $\int_{\Sing}\vp\ d\mu=0$. By Peres' Lemma (Theorem \ref{thm:Peres_Lemma}),
there exists $\v\in\Sing$ such that for all $T>0$ 
\begin{equation}\label{eqn: peres_phi}
\int_{0}^{T}\vp(f_{\tau}\v)d\tau\geq T\int_{\Sing}\vp\ d\mu=0.
\end{equation}

The second step is to create such a $(t,\delta)-$separated
set $F_{t}$ on $T^{1}M$. For $R>0$ given in the shadowing property
(3), from Lemma \ref{lem: angle_lower_bdd_reg} we get $\eta>0$ where
$w,f_{t}w\in\Reg(\eta)$ for any $w=\Pi_{t,R}(\v)$. We also obtain
the constant $L$ from (3b) corresponding to $\ep=R/2$.

We now set $\xi$ be a constant such that $\xi>\T+2L$, and for each
$N\in\N$ consider 
\[
\mathcal{A}:=\{\xi,2\xi,\ldots,(N-1)\xi\}\subseteq[0,N\xi]
\]
For any small $\alpha>0$ such that $\alpha N\in\N$, consider any
size $(\alpha N-1)$ subset 
\[
J=\{N_{1}\xi,\ldots,N_{\alpha N-1}\xi\}\subset\mathcal{A}
\]
with $N_{i}\in\{1,\ldots,N-1\}$. Setting $N_{0}=0$ and $N_{\alpha N}=N$,
such a subset can be viewed as a partition of the interval $[0,N\xi]$
into $\alpha N$ subintervals, each of length $n_{i}\xi$ for $n_{i}=N_{i}-N_{i-1}$.
We denote $\mathbb{J}_{N}^{\alpha}=\{J\subset\mathcal{A}:\ \#J=\alpha N-1\}$,
and we know $\#\mathbb{J}_{N}^{\alpha}={N-1 \choose \alpha N-1}$.

Given $J\in\mathbb{J}_{N}^{\alpha}$, we consider singular vectors
$\v_{i}^{J}:=f_{N_{i}\xi}(\v)$ and the corresponding regular orbits
$\{(w_{i}^{J},t_{i})\}$ where $w_{i}^{J}:=\Pi_{n_{i}\xi-\T,R}(\v_{i}^{J})$
and $t_{i}=n_{i}\xi-\T$ for $i=1,\cdots,\alpha N$. Using the specification
property (2), one can find a vector $w_{J}\in T^{1}M$ which shadows
orbits $\{(w_{i}^{J},t_{i})\}_{i=1}^{\alpha N-1}$ (\textcolor{black}{see
Fig}. \ref{fig:shadowing}), that is, for $i=1,\cdots,\alpha N$ 
\[
f_{N_{i-1}\xi}w_{J}\in B_{t_{i}}(w_{i}^{J},\rho).
\]
\begin{figure}
\includegraphics[scale=0.5]{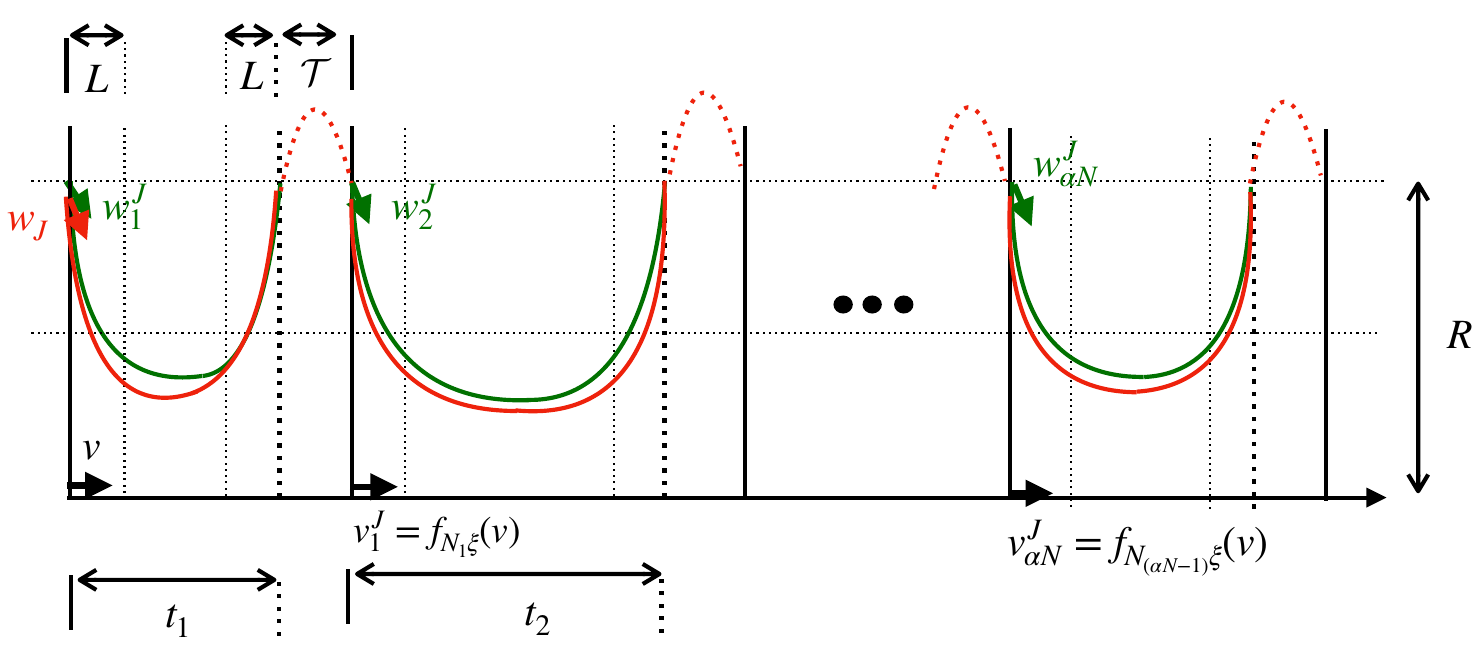}

\caption{Shadowing orbit }
\label{fig:shadowing}
\end{figure}

We claim that $F_{N\xi}:=\{(w_{J},N\xi):J\in\mathbb{J}_{N}^{\alpha}\}$
is a $(N\xi,R/2-2\rho)$-separated set. To see this, let $J_{j}=\{N_{1}^{j}\xi,\cdots,N_{\alpha N-1}^{j}\xi\}$
for $j=1,2$, and $m=\min\{n:\ N_{n}^{1}\neq N_{n}^{2}\}.$ Without
loss of generality, suppose $N_{m}^{1}<N_{m}^{2}$, then the claim
follows from the inequalities below (\textcolor{black}{see Fig. \ref{fig:separated}}
): 
\[
d(f_{N_{m}^{1}\xi -\T}(w_{J_{1}}),\Sing)>R-\rho\text{ and }d(f_{N_{m}^{1}\xi-\T}(w_{J_{2}}),\Sing)<\frac{R}{2}+\rho.
\]

\begin{figure}
\includegraphics[scale=0.5]{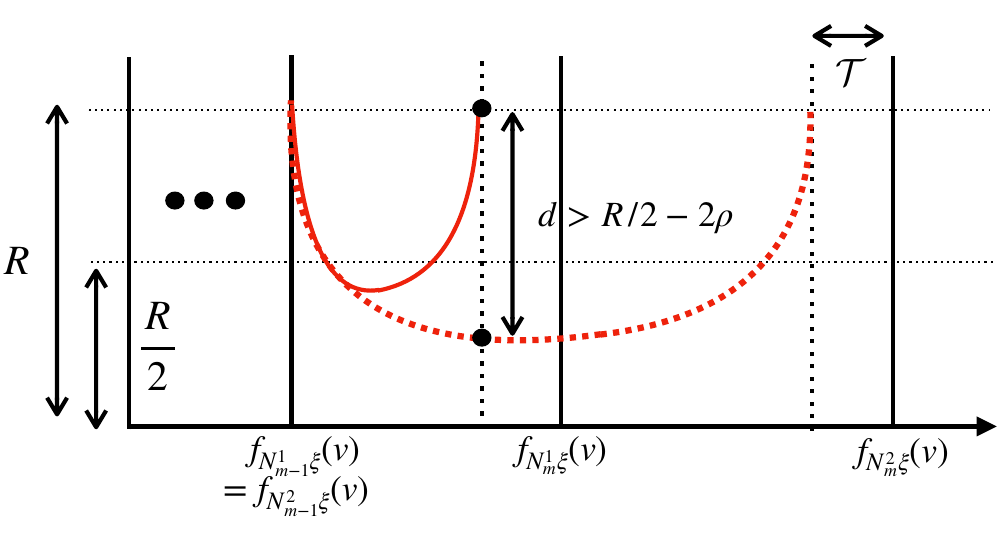}

\caption{Separated set}
\label{fig:separated}
\end{figure}
The last step is using the special Bowen property (4) to bound the
pressure from below. For $(w_{J},N\xi)\in F_{N\xi}$, the integral
of $\vp(f_{\tau}w_{J})$ during each transition period is bounded
below by $-\mathcal{T}\|\vp\|$, by the special Bowen property (4) and \eqref{eqn: peres_phi}
we get 
\begin{align}
\int_{0}^{N\xi}\vp(f_{\tau}w_{J})d\tau & =\sum_{i=1}^{\alpha N}\int_{N_{i-1}\xi}^{N_{i}\xi}\vp(f_{\tau}w_{J})d\tau\label{eqn: changes}\\
 & \geq\sum_{i=1}^{\alpha N}\int_{0}^{t_{i}}\vp(f_{\tau}\v_{i}^{J})d\tau-\alpha NC-2\alpha N||\vp||\T\nonumber \\
 & \geq\sum_{i=1}^{\alpha N}\int_{N_{i-1}\xi}^{N_{i}\xi}\vp(f_{\tau}\v_{i}^{J})d\tau-\alpha N(3\T\|\vp\|+C)\nonumber \\
 & =\int_{0}^{N\xi}\vp(f_{\tau}\v)d\tau-\alpha N(3\T\|\vp\|+C)\geq -\alpha N(3\T\|\vp\|+C).\nonumber 
\end{align}
Recall that $\#F_{N\xi}=\#\mathbb{J}_{N}^{\alpha}={N-1 \choose \alpha N-1}$,
summing the above inequality over all possible subsets $J$ gives
\begin{align*}
\sum_{w_{J}\in F_{N\xi}}e^{\int_{0}^{N\xi}\vp(f_{\tau}w_{J})d\tau} & \geq{N-1 \choose \alpha N-1}e^{-\alpha N(3\T\|\vp\|+C)}.
\end{align*}
As $F_{N\xi}$ is a $(N\xi,\delta)$-separated set (where $\delta<R/2-\rho$
with $R$ and $\rho$ are arbitrary), this implies that 
\[
{\displaystyle P(\vp)\geq\lim_{N\to\infty}\frac{1}{N\xi}\log\Big(\sum_{w_{J}\in F_{N\xi}}e^{\int_{0}^{N\xi}\vp(f_{\tau}w_{J})d\tau}\Big).}
\]
Combining the last two inequalities with ${N-1 \choose \alpha N-1}\geq\alpha e^{\left(-\alpha\log\alpha\right)N}$, one establishes
\[
P(\vphi)\geq-\frac{\alpha}{\xi}\log\alpha-\frac{\alpha (3\T\|\vp\|+C)}{\xi}.
\]
In particular, we obtain the pressure gap by choosing $\alpha$ sufficiently
small in $(0,e^{-(3\T\|\vp\|+C)})$.
\end{proof}
\begin{rem}
\label{rem: difference} \textcolor{black}{We list below several main
differences between this proof and the proof of \cite[Theorem B]{burns2018unique}.}
\begin{enumerate}
\item Our proof utilize the fact $h_{\mathrm{top}}(\mathcal{F}|_{\Sing})=0$ and Peres' Lemma (Theorem \ref{thm:Peres_Lemma}) to simplify two steps (that is \cite[Sec. 8.2 \& 8.3]{burns2018unique} in \cite[Theorem B]{burns2018unique}. Precisely, using $h_{\mathrm{top}}(\mathcal{F}|_{\Sing})=0$ and Peres' Lemma, we can easily find a vector in the singular set to shadow with regular orbit segments with good control of the integration of potentials. 
\item Equation (\ref{eqn: changes}) is another major difference. The extra constant $C$ in (\ref{eqn: changes}) comes from the condition (4) given in Section \ref{subsec: abstract result}.
This constant $C$ allows us to have some wiggling room for the potential
without losing the pressure gap.
\end{enumerate}
\end{rem}


\appendix

\section{Ricci curvature bound and radial curvature bound comparison}

{\label{sec:Appendix}}

Recall that $M$ is a $n-$dimensional manifold with nonpositive sectional
curvature, and $T_{0}\subset M$ is a totally geodesic $(n-1)$-subtorus.
We assume that $K(\sigma)=0$ for any $x\in T_{0}$ and any 2-plane
$\sigma\subset T_{x}M$. On the universal cover $\widetilde{M}$,
we define the Fermi coordinate $(s,x)$ near $\widetilde{T}_{0}$
in the following way: $s$ is the coordinate on $\widetilde{T}_{0}$,
and $x$ measures the signed distance on $\widetilde{M}$ to $\widetilde{T}_{0}$.
$s=\text{const.}$ is always a geodesic perpendicular to $\widetilde{T}_{0}$.
The Riemannian metric near $\widetilde{T}_{0}$ is 
\begin{equation}
g=dx^{2}+g_{x},\quad|x|\leq\ep\label{eqn: metric_HD-1}
\end{equation}
where $g_{x}$ is the Riemannian metric on $\widetilde{T}_{x}:=\widetilde{T}_{0}\times\{x\}$.
In particular, $g_{0}$ is the Euclidean metric on $\widetilde{T}_{0}.$

If $g_{x}$ is a warped product, namely, $g=dx^{2}+f(x)^{2}g_{0}$.
We have $f(0)=1$. Since $T_{0}$ is totally geodesic, we have $f'(0)=0$. 
\begin{lem}\label{lem: warped_product_curv}
If $g=dx^{2}+f(x)^{2}g_{0}$, then the following conditions are equivalent: 

\begin{enumerate}
\item There exists $C_{1},C_{2},\ep>0$ such that 
\[
-C_{1}|x_{v}|^{m}\leq K_{\perp}(v)\leq-C_{2}|x_{v}|^{m},\text{for any \ensuremath{v\perp X} with }|x_{v}|<\ep.
\]
\item There exists $C_{1},C_{2},\ep>0$ so that 
\[
-C_{1}|x_{v}|^{m}\leq\text{Ric}(v)\leq-C_{2}|x_{v}|^{m},\text{for any \ensuremath{v} with }|x_{v}|<\ep.
\]
\item There exists $C_{1},C_{2},\ep>0$ so that 
\[
C_{1}|x|^{m+2}\leq f(x)-1\leq C_{2}|x|^{m+2}.
\]
\end{enumerate}
\end{lem}

\begin{proof}
Since $g=dx^{2}+f(x)^{2}g_{0}$, the radial curvature is 
\[
K_{\perp}(v)=-\frac{f''(x_{v})}{f(x_{v})},
\]
while the sectional curvature is given by 
\[
K_{\sigma}=-\frac{f''}{f}\cos^{2}\theta-\left(\frac{f'}{f}\right)^{2}\sin^{2}\theta,
\]
where $\theta$ is the angle between $X$ and $\sigma$.

Now we compute the Ricci curvature. If $v$ is normal, then $\text{Ric}(v)=(n-1)K_{\perp}(v)$
and we are done. Otherwise, denote by $v^{\perp}$ the perpendicular
complement in $T_{p}M$, $\theta$ the angle between $v$ and $X$,
and $H(v)$ the horizontal subspace at $v$. Since both $v^{\perp}$
and $H(v)$ have codimension 1, $v^{\perp}\cap H(v)$ has dimension
$n-2$. We can construct an orthonormal basis $\{e_{i}\}_{i=1}^{n}$
such that $e_{1}=v$, $e_{3},\cdots,e_{n}\in v^{\perp}\cap H(v)$,
and the section spanned by $e_{1},e_{2}$ is normal. Then we have
\begin{equation}
\text{Ric}(v)=-\frac{f''}{f}+(n-2)\left(-\frac{f''}{f}\cos^{2}\theta-\left(\frac{f'}{f}\right)^{2}\sin^{2}\theta\right)\label{eqn: ric_curv}
\end{equation}
When $\theta=0$, we get $\text{Ric}(v)=(n-1)K_{\perp}(v)$. Thus,
for any $v$, the Ricci curvature can be calculated using \eqref{eqn: ric_curv}.

$(1)\Rightarrow(2):$ Since $K_{\perp}=-f''/f$ and $f(0)=1$, we
have $f''\approx|x|^{m}$. Since $f'(0)=0$, $f'\approx|x|^{m+1}$.
Therefore, $\text{Ric}(v)\approx-|x_{v}|^{m}$ by \eqref{eqn: ric_curv}.

$(2)\Rightarrow(3):$ Assume $f-1\approx|x|^{k}$ for some $k>2$.
We have $(f'/f)^{2}\approx|x|^{2k-2}$ and $f''/f\approx|x|^{k-2}$.
Thus, $f''/f$ is the dominant term in \eqref{eqn: ric_curv}. Therefore,
$k=m+2$.

$(3)\Rightarrow(1):$ Since $f-1\approx x^{m+2}$ and $f'(0)=1$,
we have $f'\approx x^{m+1}$ and $f''\approx x^{m}$, thus $K_{\perp}=-f''/f\approx-|x|^{m}$. 
\end{proof}

\section{A lemma of Peres}

In this section, we prove a proof of Theorem \ref{thm:Peres_Lemma}
as Peres' original theorem \cite[Lemma 2]{Peres88} is for transformations.
\begin{thm}
\cite[Lemma 2]{Peres88} \label{thm:proof-peres}Let $\mathcal{F}=\{f_{t}\}$
be a continuous flow on a compact space $X$, and $\mu$ be an $\mathcal{F}-$invariant
probability measure. Then for every potential $\vp:X\to\mathbb{R}$
there exists some $v\in X$ 
\[
\frac{1}{T}\int_{0}^{T}\vp(f_{\tau}v)d\tau\geq \int_{X}\vp d\mu
\]
for all $T>0$. 
\end{thm}

Peres' proof is based on the Maximal Ergodic Theorem, and it works
almost line by line in the flow case. A version of the Maximal Ergodic
Theorem for flows can be found in \cite[Theorem 1.2, P.76]{Petersen83book}.
\begin{thm}[Maximal Ergodic Theorem]
Let $(X,\mathcal{B},\mu)$ be a probability space and $\mathcal{F}=\{f_{t}\}$
be a measure-preserving flow on $X$. If $\vp\in L^{1}(\mu)$ and
$\alpha\in\mathbb{R}$, then 
\[
\int_{\{\vp^{*}>\alpha\}}{\vp d\mu}\geq\alpha\cdot\mu\{\vp^{*}>\alpha\}
\]
where ${\displaystyle {\vp^{*}(x)=\sup_{T>0}\frac{1}{T}\int_{0}^{T}{\vp(f_{\tau}(x))d\tau}}.}$
\end{thm}

\begin{proof}[Proof of Theorem \ref{thm:proof-peres}]
For $\epsilon>0$, we define 
\[
E_{\epsilon}:=\{x\in X:\ \forall t\geq0,\ \frac{1}{t}\int_{0}^{t}\vp(f_{\tau}x)d\tau>\int_{X}\vp d\mu-\ep\}
\]
and 
\[
\Psi(x):=\int_{X}\vp d\mu-\vp-\epsilon.
\]
It is enough to show 
\[
\bigcap_{\epsilon>0}E_{\epsilon}\neq\emptyset.
\]
To see this, we first notice that $E_{\epsilon}=\{x:\ \Psi^{*}(x)\leq0\}$.
We then apply the Maximal Ergodic Theorem on $\Psi$ and $X\backslash E_{\epsilon}=\{x:\ \Psi^{*}(x)>0\}$,
i.e., $\alpha=0$, and get 
\[
\int_{X\backslash E_{\epsilon}}\Psi d\mu\geq0.
\]
Observe that $\int_{X}\Psi d\mu=-\epsilon$, and which guarantees
$E_{\epsilon}\neq\emptyset$ for all $\epsilon>0$. Since $\vp$ is
continuous and $X$ is compact, we know $E_{\epsilon}$'s are nested
compact sets, and thus 
\[
\bigcap_{\epsilon>0}E_{\epsilon}\neq\emptyset.
\]
\end{proof}

\bibliographystyle{amsalpha}
\bibliography{pgap_new}

\newcommand{\etalchar}[1]{$^{#1}$}
\providecommand{\bysame}{\leavevmode\hbox to3em{\hrulefill}\thinspace}
\providecommand{\MR}{\relax\ifhmode\unskip\space\fi MR }
\providecommand{\MRhref}[2]{%
  \href{http://www.ams.org/mathscinet-getitem?mr=#1}{#2}
}
\providecommand{\href}[2]{#2}
\begin{thebibliography}{CLMT19}

\bibitem[ALP24]{Lima2020Bernulli}
Ermerson Araujo, Yuri Lima, and Mauricio Poletti, \emph{Symbolic dynamics for
  nonuniformly hyperbolic maps with singularities in high dimension}, Mem.
  Amer. Math. Soc. \textbf{301} (2024), no.~1511, vi+117.

\bibitem[Bal95]{ballmann1995lectures}
Werner Ballmann, \emph{Lectures on spaces of nonpositive curvature}, vol.~25,
  Springer Science \& Business Media, 1995.

\bibitem[BBE85]{ballman1985nonpositive}
Werner Ballmann, Misha Brin, and Patrick Eberlein, \emph{Structure of manifolds
  of nonpositive curvature. {I}}, Ann. of Math. (2) \textbf{122} (1985), no.~1,
  171--203.

\bibitem[BBFS21]{burns2021phase}
Keith Burns, J{\'e}r{\^o}me Buzzi, Todd Fisher, and Noelle Sawyer, \emph{Phase
  transitions for the geodesic flow of a rank one surface with nonpositive
  curvature}, Dynamical Systems \textbf{36} (2021), no.~3, 527--535.

\bibitem[BCFT18]{burns2018unique}
Keith Burns, Vaughn Climenhaga, Todd Fisher, and Daniel~J Thompson,
  \emph{Unique equilibrium states for geodesic flows in nonpositive curvature},
  Geometric and Functional Analysis \textbf{28} (2018), no.~5, 1209--1259.

\bibitem[BG89]{Burns1989Bernoulli}
Keith Burns and Marlies Gerber, \emph{Real analytic {B}ernoulli geodesic flows
  on {$S^2$}}, Ergodic Theory Dynam. Systems \textbf{9} (1989), no.~1, 27--45.

\bibitem[BG14]{burns2014lyapunov}
Keith Burns and Katrin Gelfert, \emph{Lyapunov spectrum for geodesic flows of
  rank 1 surfaces}, Discrete \& Continuous Dynamical Systems \textbf{34}
  (2014), no.~5, 1841--1872.

\bibitem[Bow74]{bowen1974some}
Rufus Bowen, \emph{Some systems with unique equilibrium states}, Mathematical
  systems theory \textbf{8} (1974), no.~3, 193--202.

\bibitem[CCE{\etalchar{+}}23]{Call2021flatsurfaces}
Benjamin Call, David Constantine, Alena Erchenko, Noelle Sawyer, and Grace
  Work, \emph{Unique equilibrium states for geodesic flows on flat surfaces
  with singularities}, Int. Math. Res. Not. IMRN (2023), no.~17, 15155--15206.

\bibitem[CKP20]{chen2020unique}
Dong Chen, Lien-Yung Kao, and Kiho Park, \emph{Unique equilibrium states for
  geodesic flows over surfaces without focal points}, Nonlinearity \textbf{33}
  (2020), no.~3, 1118--1155.

\bibitem[CKP21]{chen2021properties}
\bysame, \emph{Properties of equilibrium states for geodesic flows over
  manifolds without focal points}, Advances in Mathematics \textbf{380} (2021),
  107564.

\bibitem[CKW21]{Climenhaga2021noconjugate}
Vaughn Climenhaga, Gerhard Knieper, and Khadim War, \emph{Uniqueness of the
  measure of maximal entropy for geodesic flows on certain manifolds without
  conjugate points}, Adv. Math. \textbf{376} (2021), Paper No. 107452, 44.

\bibitem[CLMT19]{MR3961216}
D.~Constantine, J.-F. Lafont, D.~B. McReynolds, and D.~J. Thompson, \emph{Fat
  flats in rank one manifolds}, Michigan Math. J. \textbf{68} (2019), no.~2,
  251--275. \MR{3961216}

\bibitem[CS14]{Coudene2014}
Yves Coud\`ene and Barbara Schapira, \emph{Generic measures for geodesic flows
  on nonpositively curved manifolds}, J. \'{E}c. polytech. Math. \textbf{1}
  (2014), 387--408. \MR{3322793}

\bibitem[CT16]{Climenhaga2016Uniqueness}
Vaughn Climenhaga and Daniel~J. Thompson, \emph{Unique equilibrium states for
  flows and homeomorphisms with non-uniform structure}, Adv. Math. \textbf{303}
  (2016), 745--799.

\bibitem[CT21]{Climengaha2021survey}
\bysame, \emph{Beyond {B}owen's specification property}, Thermodynamic
  formalism, Lecture Notes in Math., vol. 2290, Springer, Cham, 2021,
  pp.~3--82.

\bibitem[CT22]{call2022K}
Benjamin Call and Daniel~J. Thompson, \emph{Equilibrium states for
  self-products of flows and the mixing properties of rank 1 geodesic flows},
  J. Lond. Math. Soc. (2) \textbf{105} (2022), no.~2, 794--824.

\bibitem[CX08]{cao2008flatstrip}
Jianguo Cao and Frederico Xavier, \emph{A closing lemma for flat strips in
  compact surfaces of non-positive curvature}, preprint.

\bibitem[Ebe01]{eberlein2001geodesic}
Patrick Eberlein, \emph{Geodesic flows in manifolds of nonpositive curvature},
  Proceedings of Symposia in Pure Mathematics, vol.~69, Providence, RI;
  American Mathematical Society; 1998, 2001, pp.~525--572.

\bibitem[EH90]{eschenburg1990comparison}
J-H Eschenburg and Ernst Heintze, \emph{Comparison theory for riccati
  equations}, Manuscripta mathematica \textbf{68} (1990), no.~1, 209--214.

\bibitem[Fra77]{franco1977flows}
Ernesto Franco, \emph{Flows with unique equilibrium states}, American Journal
  of Mathematics \textbf{99} (1977), no.~3, 486--514.

\bibitem[GN99]{gerber1999ETDS}
Marlies Gerber and Viorel Ni\c{t}ic\u{a}, \emph{H\"{o}lder exponents of
  horocycle foliations on surfaces}, Ergodic Theory Dynam. Systems \textbf{19}
  (1999), no.~5, 1247--1254.

\bibitem[GR19]{Gelfert2019mme}
Katrin Gelfert and Rafael~O. Ruggiero, \emph{Geodesic flows modelled by
  expansive flows}, Proc. Edinb. Math. Soc. (2) \textbf{62} (2019), no.~1,
  61--95.

\bibitem[GS14]{Gelfert:2014hn}
Katrin Gelfert and Barbara Schapira, \emph{Pressures for geodesic flows of rank
  one manifolds}, Nonlinearity \textbf{27} (2014), no.~7, 1575--1594.
  \MR{3225873}

\bibitem[GW99]{gerber1999holder}
Marlies Gerber and Amie Wilkinson, \emph{H\"older regularity of horocycle
  foliations}, Journal of Differential Geometry \textbf{52} (1999), no.~1,
  41--72.

\bibitem[Kni98]{knieper1998uniqueness}
Gerhard Knieper, \emph{The uniqueness of the measure of maximal entropy for
  geodesic flows on rank 1 manifolds}, Annals of mathematics (1998), 291--314.

\bibitem[LLS16]{ledrappier2016Bernulli}
Fran\c{c}ois Ledrappier, Yuri Lima, and Omri Sarig, \emph{Ergodic properties of
  equilibrium measures for smooth three dimensional flows}, Comment. Math.
  Helv. \textbf{91} (2016), no.~1, 65--106.

\bibitem[LMM24]{lima2021polydecay}
Yuri Lima, Carlos Matheus, and Ian Melbourne, \emph{Polynomial decay of
  correlations for nonpositively curved surfaces}, Trans. Amer. Math. Soc.
  \textbf{377} (2024), no.~9, 6043--6095.

\bibitem[LS19]{Lima-Sarig19-3d-symb}
Yuri Lima and Omri~M. Sarig, \emph{Symbolic dynamics for three-dimensional
  flows with positive topological entropy}, J. Eur. Math. Soc. (JEMS)
  \textbf{21} (2019), no.~1, 199--256.

\bibitem[OW98]{Ornstein1998Bernoulli}
Donald Ornstein and Benjamin Weiss, \emph{On the {B}ernoulli nature of systems
  with some hyperbolic structure}, Ergodic Theory Dynam. Systems \textbf{18}
  (1998), no.~2, 441--456.

\bibitem[Per88]{Peres88}
Yuval Peres, \emph{A combinatorial application of the maximal ergodic theorem},
  Bull. London Math. Soc. \textbf{20} (1988), no.~3, 248--252.

\bibitem[Pes77]{Pesin1977Bernoulli}
Ja.~B. Pesin, \emph{Characteristic {L}japunov exponents, and smooth ergodic
  theory}, Uspehi Mat. Nauk \textbf{32} (1977), no.~4 (196), 55--112, 287.

\bibitem[Pet83]{Petersen83book}
Karl Petersen, \emph{Ergodic theory}, Cambridge Studies in Advanced
  Mathematics, vol.~2, Cambridge University Press, Cambridge, 1983. \MR{833286}

\bibitem[Rue78]{Ruelle1978Diff}
David Ruelle, \emph{An inequality for the entropy of differentiable maps}, Bol.
  Soc. Brasil. Mat. \textbf{9} (1978), no.~1, 83--87.

\bibitem[TW21]{Thompson2021CLT}
Daniel~J. Thompson and Tianyu Wang, \emph{Fluctuations of time averages around
  closed geodesics in non-positive curvature}, Comm. Math. Phys. \textbf{385}
  (2021), no.~2, 1213--1243.

\end{thebibliography}

\end{document}